\documentclass[a4paper,11pt]{article}
\usepackage[top=3.0cm, bottom=3.0cm, inner=3.0cm, outer=3.0cm, includefoot]{geometry}
\usepackage[latin1]{inputenc}
\usepackage[T1]{fontenc}
\usepackage{verbatim}

\usepackage{geometry}
\usepackage{amssymb}
\usepackage{amsmath}
\usepackage{graphicx}
\usepackage{amsthm}

\usepackage{color}

\usepackage{enumerate}

\newcommand{\IR}{{\mathbb{R}}}

\newcommand{\IN}{{\mathbb{N}}}

\newcommand{\dd}{n}
\newcommand{\LB}{\Delta}
\newcommand{\posR}{\mathbb{R}^+}
\newcommand{\nnegR}{\mathbb{R}^+_0}
\newcommand{\sol}{u \in C^1( V \times \nnegR)}
\newcommand{\downto}{{\searrow}}
\newcommand{\eChar}{\begin{enumerate}[(i)]}
\newcommand{\eBr}{\begin{enumerate}[(1)]}
\newcommand{\ii}{j(i)}
\newcommand{\zi}{z_{i'}}
\newcommand{\LL}{\mathcal L}
\newcommand{\II}{\posR}
\newcommand{\ddx}{\frac{d}{dx}}

\newcommand{\Abstract}
{
We introduce a new version of a curvature-dimension inequality for non-negative curvature. We use this inequality to prove a logarithmic Li-Yau inequality on finite graphs.
To formulate this inequality, we introduce a non-linear variant of the calculus of Bakry and Émery. In the case of manifolds, the new calculus and the new curvature-dimension inequality coincide with the common ones. 
In the case of graphs, they coincide in a limit.
In this sense, the new curvature-dimension inequality gives a more general concept of curvature on graphs and on manifolds. We show that Ricci-flat graphs have a non-negative curvature in this sense. Moreover, a variety of non-logarithmic Li-Yau type gradient estimates can be obtained by using the new Bakry-Émery type calculus. Furthermore, we use these Li-Yau inequalities to derive Harnack inequalities on graphs.
}

\title
{
Li-Yau inequality on finite graphs via non-linear curvature dimension conditions
}

\author{Florentin Münch}
\date{\today}

\theoremstyle{plain}
\newtheorem{lemma}{Lemma}[section]
\newtheorem{theorem}[lemma]{Theorem}
\newtheorem{proposition}[lemma]{Proposition}
\newtheorem{corollary}[lemma]{Corollary}

\theoremstyle{definition}

\newtheorem{example}[lemma]{Example}

\newtheorem{rem}[lemma]{Remark}
\newtheorem{defn}[lemma]{Definition}

\numberwithin{equation}{section}

\begin{document}

\maketitle

\begin{abstract}
\Abstract
\end{abstract}
\tableofcontents




\pagestyle{plain}


\section{Introduction}


  
In 1986, Li and Yau proved a gradient estimate for manifolds, later known  as Li-Yau inequality (cf. \cite{Li1986}). This states that for positive solutions $u$ to the heat equation $\LL(u) = 0$ (with $\LL = \Delta - \partial_t$) on a $d$-dimensional compact manifold without boundary, the following implication holds.
\begin{equation}
Ricc \geq 0 \Longrightarrow - \Delta \log u (x,t) \leq \frac d {2t} \label{LY86}.
\end{equation}
The term $Ricc$ denotes the Ricci curvature of the manifold.
An important application of this inequality is the Harnack inequality (diffusion) which can be seen as an integrated form of the Li-Yau inequality.
The different forms of the Harnack inequality (logarithmic, harmonic, diffusion) are closely related.

In 2006, Bakry and Ledoux generalized the Li-Yau inequality result to more general Laplacians that satisfy the chain rule and improved the result by using a curvature-dimension inequality ($CD$-inequality) instead of the non-negative Ricci curvature and by giving a characterization of the $CD$-inequality via a logarithmic Sobolev inequality containing the Li-Yau inequality (cf. \cite{Bakry2006}). The $CD$-inequality will be introduced in the next subsection.

After Li's and Yau's breakthrough in 1986, great effort was made to establish an analog result on graphs. 
This turned out to be very complicated since all known proofs of the Li-Yau inequality had made an extensive use of the chain rule, but this chain rule does not hold on graphs.
As a first step to consider graphs with non-negative curvature, Chung and Yau introduced the concept of Ricci-flat graphs in 1996 (cf. \cite{Chung1996}) and obtained the following results for graphs.
\begin{equation}
d \mbox{-Ricci-flat} \Longrightarrow \mbox{Harnack inequality (logarithmic)}. \nonumber
\end{equation}
Ricci-flat graphs are a slight generalization of Abelian Cayley graphs (cf. Subsection \ref{SRF}). It remains unclear, whether there are Ricci-flat graphs which are not Abelian Cayley graphs.

In 2012, Chung, Lin, and Yau found a first Li-Yau type result using the $CD(d,0)$ inequality on graphs (cf. \cite{Chung2012}). They showed for graphs
\begin{eqnarray}
d\mbox{-Ricci-flat} &\Longrightarrow& CD(d,0) \\
									  &\Longrightarrow&  \mbox{ Harnack inequality (harmonic). }  \nonumber
\end{eqnarray}

One year later, in 2013, Bauer, Horn, Lin, Lippner, Mangoubi, and Yau proved a result on graphs which is very similar to the original Li-Yau inequality from 1986 (cf. \cite{Bauer2013}). This result is the following. If $G$ is a graph and $\sol$ is a positive solution to the heat equation on $G$, then
\begin{eqnarray}
d\mbox{-Ricci-flat} &\Longrightarrow&   CDE(d,0) \\
										&\Longrightarrow&   \frac {\Gamma(\sqrt u)}{u} - \frac {\partial_t \sqrt u}{\sqrt u} \leq \frac d {2t} \label{BLY} \\
										&\Longrightarrow&   \mbox{Harnack inequality (diffusion)}  . \nonumber
\end{eqnarray}
The $CDE(d,0)$ condition is the exponential curvature-dimension inequality, a substitute of $CD(d,0)$. 
The gradient form $\Gamma$ has been introduced by Bakry and Émery (cf. \cite{Bakry1985}). We will define this in the next subsection.
There are many generalizations of this statement. In \cite{Bauer2013}, general curvature bounds and potentials have been discussed. In \cite{Qian2013}, the gradient estimate have been proven with time dependent coefficients.
Obviously, the gradient estimate (\ref{BLY}) which we will call the $\sqrt{\cdot}$ Li-Yau inequality, has a different form than the logarithmic Li-Yau inequality (\ref{LY86}). Additionally, we can see in the following that in the $CDE$ inequality, there is a break of analogy to the original result on manifolds. In the examples section (Section \ref{CExamples}), we will explain this break of analogy and why it occurs. 
They overcome the missing chain rule in a remarkable way by observing that a version of the chain rule for the square root surprisingly also holds on graphs.

In this article, we prove an inequality on graphs which has the same form as the logarithmic Li-Yau inequality (\ref{LY86}) from 1986. 
To do so, we introduce a new version of the $CD$ inequality to avoid the use of the chain rule (cf. Subsection \ref{SCD}).
We will show
\begin{eqnarray}
 d C_{\log}  \mbox{-Ricci-flat} &\Longrightarrow&  CD \log(d,0) \\ 
 															&\Longrightarrow&  - \Delta \log u (x,t) \leq \frac d {2t} \label{logLY} \\ 
 															&\Longrightarrow&  \mbox{Harnack inequality (diffusion)},      \nonumber
\end{eqnarray}
where $\sol$ is a positive solution to the heat equation, and $C_{\log}$ is a positive constant. The $CD \log(d,0)$ inequality is a new curvature-dimension condition on graphs. If there is a chain rule for the Laplacian as in the case of manifolds, the $CD \log(d,0)$ inequality is equivalent to the $CD (d,0)$ inequality (cf. Subsection \ref{SMF}).
In the case of graphs, the $CD \log(d,0)$ inequality implies the $CD (d,0)$ inequality (cf. Subsection \ref{SLT}).

When establishing our results, we even introduce a more general concept by replacing the logarithm by a concave function  $\psi \in C^1(\posR)$. This leads us to the non-linear Laplacian $\Delta^\psi$ (cf. Subsection \ref{SDpsi}) as a substitution for the expression $\Delta \log (\cdot)$. With the choice $\psi=\sqrt{\cdot}$ respectively $\psi=\log$, we obtain the logarithmic respectively the $\sqrt{\cdot}$ Li-Yau inequality (cf. Example \ref{Elog}).

As discussed above, one goal of this article is to use these Li-Yau estimates to deduce Harnack-inequalities which have the form
\[
\frac{u(x,T_1)}{u(y,T_2)} \leq C(d(x,y),T_1,T_2)
\]

with a positive solution $\sol$ to the heat equation, positive real numbers $T_1 < T_2$, and a constant $C(d(x,y),T_1,T_2)$ depending only on the distance of $(x,T_1)$ and $(y,T_2)$ in space-time.

It turns out that the Harnack-estimates presented here are stronger than the ones of Bauer, Horn, Lin, Lippner, Mangoubi, and Yau in \cite{Bauer2013} (cf. Corollary \ref{CHR}).

In this article, we use finite graphs as the basis of our discrete setting. In future work, we plan to also consider infinite graphs or even Dirichlet forms as a more general setting.

{\bf Acknowledgements}

I wish to thank Matthias Keller and Daniel Lenz, who have been the supervisors of my master thesis which is presented here in this article, for many useful discussions and for providing a very enjoyable and constructive atmosphere. Moreover, I wish to acknowledge Matthias Keller for proposing the topic of my master thesis. 

\section{Basics}

We are interested in giving discrete analoga of the so called Li-Yau inequality on manifolds. 

\begin{defn}[Graph]
A pair $G=(V,E)$ with a finite set $V$ and a relation $E \subset V\times V$ is called a \emph{finite graph} if $(v,v) \notin E$ for all $v \in V$ and if $(v,w) \in E$ implies $(w,v) \in E$ for $v,w \in V$.

For $v,w \in V$, we write $v \sim w$ if $(v,w) \in E$. In this case, we say that the vertices $v$ and $w$ are \emph{adjacent}.
For $v\in V$, we denote $\deg v := \#\{w \in V : w \sim v\}$. 
\end{defn}

\begin{defn}[Path metric on graphs]
Let $G=(V,E)$ be a finite graph. 
A sequence $(v_0,\ldots,v_n) \in V^{n+1}$ is called a \emph{path} of the length $n$ from $v_0$ to $v_n$ if all $v_i$  are pairwise distinct  and $v_i \sim v_{i-1}$ for all $i \in \{1,\ldots,n\}$.
We call $d:V \times V \to [0,\infty]$,
\[
d(x,y):= \inf \{n \in \IN: \mbox{there is a path of the length $n$ from $x$ to $y$}  \}
\]
the \emph{distance} of $x$ and $y$.
The finite graph $G$ is called \emph{connected} if there is a path from $x$ to $y$ for all $x,y \in V$ .
\end{defn}

A well studied Laplacian on manifolds is the Laplace-Beltrami operator. On graphs, there is an analogon of this operator. 

\begin{defn}[Laplacian $\Delta$]
Let $G=(V,E)$ be a finite graph.
The domain of the Laplacian $\Delta$ is
\[
C(V):= \IR^V := \{f: V \to \IR\ \}.
\]
The \emph{Laplacian} $\Delta : C(V) \to C(V)$ is defined for $f \in C(V)$ and $v \in V$ as
\[
\Delta f (v) := \sum_{w \sim v} (f(w) - f(v)).
\]
\end{defn}

\begin{rem}
Following the definition of Li and Yau in \cite{Li1986}, we deal with the Laplacian with a negative sign.
\end{rem}

We introduce the $\Gamma$-calculus by Bakry and Émery (cf. \cite{Bakry1985}). Note that such a calculus can be defined whenever a sufficiently nice Laplacian is given. 
Especially, it is useful on manifolds and on graphs. On manifolds, this calculus has been studied e.g. in \cite{Bakry1985, Bakry1992, Bakry2006}, and on graphs, it has been studied in \cite{Bauer2012,Bauer2013,Chung2012,Jost2014,Lin2010,Lin2011}. 

\begin{defn}[$\Gamma$-calculus] \label{defGamma}
Let $G=(V,E)$ be a finite graph.
Then, the \emph{gradient form} or \emph{carré du champ} operator $\Gamma : C(V) \times C(V) \to C(V)$ is defined by
\[
2 \Gamma (f,g) := \Delta(fg) - f\Delta g - g\Delta f.
\]
Similarly, the \emph{second gradient form} $\Gamma_2 : C(V) \times C(V) \to C(V)$ is defined by
\[
2 \Gamma_2 (f,g) := \Delta \Gamma (f, g) - \Gamma(f, \Delta g)  - \Gamma (g, \Delta f).
\]
We write $\Gamma (f):= \Gamma (f,f)$ and $\Gamma_2 (f):= \Gamma_2 (f,f)$.
\end{defn}

On manifolds, there is a nice relation between its Ricci curvature and the second gradient form which is a consequence of Bochner's formula (cf. \cite{Bochner1953}). This relation is
\[
\Gamma_2(f) \geq \frac 1 d (\Delta f)^2 + Ricc(\nabla f)
\]
for all $f$, where $d$ is the dimension of the manifold and $Ricc$ is the Ricci curvature.
Especially, if $Ricc \geq 0$, then
\[
\Gamma_2 (f) \geq \frac 1 d (\Delta f)^2.
\]
This motivates to introduce a curvature-dimension inequality on graphs, where no suitable explicit definition of the Ricci curvature is known yet.

\begin{defn}
We write $\posR := (0,\infty)$ and $\nnegR := [0,\infty)$.

Let $G=(V,E)$ be a finite graph. Then, we write
\[
C^+(V) := \{f:V \to \posR\}.
\]
\end{defn}

\begin{defn}[$CD(d,0)$ condition] \label{DCD}
Let $G=(V,E)$ be a finite graph and $d \in \posR$.
We say $G$ satisfies the \emph{curvature-dimension inequality} $CD(d,0)$ if for all $f \in C(V)$, 
\[
\Gamma_2(f) \geq \frac 1 d (\Delta f)^2.
\]
We can interpret this as meaning that that the graph $G$ has a dimension (at most) $d$ and a non-negative Ricci curvature.
\end{defn}

\begin{rem}
This curvature-dimension inequality has been studied on manifolds e.g. by Bakry and Émery in \cite{Bakry1985, Bakry1992, Bakry2006}. On graphs, it has been studied in \cite{Bauer2012,Bauer2013,Chung2012,Jost2014,Lin2010,Lin2011}. There is a great interest in giving generalizations of the Ricci curvature. Apart from the Bakry Émery approach, there is a well studied concept of Ricci curvature on metric measure spaces via optimal transport (cf. \cite{Lott2009,Ollivier2009,Ollivier2010,Sturm2006}).
Connections between these different approaches are given in \cite{Ambrosio2012, Jost2014}. A notion of Ricci curvature on cell complexes by counting neighboring cells is presented in \cite{Forman2003}.
\end{rem}

\begin{defn}[$C^1(I)$ for intervals $I$]
Let $I \subseteq \IR$ be a (not necessarily open) interval and let $\phi : I \to \IR$ be a function. We call $\phi$ (continuously) differentiable if $\phi$ can be 
extended to a (continuously) differentiable function on an open interval $J$ with $I \subset J \subset \IR$. The derivative of $\phi$ is the derivative of that extension of $\phi$ restricted to $I$.
We write 
\begin{eqnarray*}
C^1(I) &:=& \{\phi:  I \to \IR \; | \; \phi \mbox { is continuously differentiable } \},  \\
C(I)   &:=& \{\phi:  I \to \IR \; | \; \phi \mbox { is continuous } \}.
\end{eqnarray*}
\end{defn}

We will be mostly interested in the case $I \in \{\IR, \posR, \nnegR \}$.
Since the Li-Yau inequality deals with solutions to the heat equation, we define the heat operator.

\begin{defn}[Heat operator $\LL$]
Let $G=(V,E)$ be a graph.
The domain of the heat operator $\LL$ is
\[
C^1(V \times \II) := \{u: V \times \II \to \IR \; | \; u \mbox { is continuously differentiable in the second variable} \}.
\]
For $u \in C^1(V \times \II)$ we write
\[
u_t(v):=u(v,t)
\]
for all $v \in V$ and $t \in \II$. 
We call $t \in \II$ the \emph{time}, and we call $v \in V$ the \emph{location} respectively \emph{position}.

The range of the heat operator is
\[
C(V \times \II) := \{u: V \times \II \to \IR \; | \;  u \mbox { is continuous in the second variable} \}.
\]
The \emph{heat operator} is defined by $\LL(u) := \Delta u  - \partial_t u$ for all $u \in C^1(V \times \II)$.

We call a function $\sol$ a \emph{solution to the heat equation} on $G$ if  $\LL(u) = 0$. 
\end{defn}

Later we will also use the heat operator on $C^1(V \times I)$ for an interval $I \subset \IR$. The definition above can be extended to this case in a natural way.

Most proofs of the Li-Yau inequality on manifolds make an extensive use of the chain rule for the Laplace-Beltrami operator given in the next proposition.
\begin{proposition}[Chain rule on manifolds]
Let $\LB$ be the Laplace-Beltrami operator on a complete, connected manifold $M$ (with its Riemannian measure) and $\Phi \in C^\infty( \posR)$ and $f,g \in C^\infty (M)$ with $f>0$. Then, the chain rules 
\begin{eqnarray}
\LB \Phi(f) 		&=& \Phi'(f)  \LB f + \Phi''(f)      			\Gamma (f) \label{CR}, \\
\Gamma(\Phi(f),g) &=& \Phi'(f) \Gamma (f,g)               \label{CRG},\\
\Gamma(f,gh)      &=& h \Gamma (f,g) + g \Gamma (f,h)     \label{CRP}
\end{eqnarray}
are valid.
\end{proposition}
\begin{proof}
A proof of this claim is given in \cite{Bakry1992}.
\end{proof}
\begin{rem}
There is an established theory of more general Laplace operators which satisfy the three equations above.
These operators are said to be diffusion operators (cf. \cite{Bakry1985, Bakry1992, Bakry2006}).
\end{rem}

Unfortunately, there is no chain rule for the Laplacian on graphs. So, we have to find a way to bypass the chain rule, as we want to prove the Li-Yau inequality.
We will use a key identity for $\Gamma_2$ in the case of manifolds to introduce a new second gradient by which we can avoid the use of the chain rule. This identity states that for all solutions $\sol$ to the heat equation, one has
\[
\LL(u \Delta (\log u)) = -2 u\Gamma_2 (\log u).
\]
Note that this statement is not true in general if there is no chain rule. But in the next section, we will define a scaling invariant Laplacian $\Delta^\psi$ and a second gradient $\Gamma_2^\psi$, such that
\[
\LL(u \Delta^\psi u)) = -2 u\Gamma_2^\psi (u)
\]
holds for all solutions $\sol$ to the heat equation and no chain rule is needed.

\section{The $\Gamma^\psi$-calculus}

We start this section with a short summary of its contents:

As the basis for our new calculus, we will introduce the operator $\Delta^\psi$ as a scaling invariant and non-linear generalization of the Laplacian with the parameter 
$$\psi \in C^1(\posR) =\{\phi: \posR \to \IR   \mbox{ | }  \phi' \mbox{ is continuous} \}.$$
Using this, we introduce new gradients $\Gamma^\psi$ and $\Gamma_2^\psi$.

Let $G=(V,E)$ be a finite graph.
Our goal is to generalize the Li-Yau inequality to a $\psi$-Li-Yau inequality with the form
\begin{equation}
-\Delta^\psi u \leq \frac d {2t} \label{psiLY}
\end{equation}
for all positive solutions $\sol$ to the heat equation which satisfy the $CD\psi(d,0)$ inequality which states
\[
\Gamma_2^\psi (f) \geq \frac 1 d (\Delta^\psi f)^2
\]
for all positive $f \in C(V)$. The second $\psi$-gradient $\Gamma_2^\psi$ will be defined in the $CD\psi$ subsection (Subsection \ref{SCD}).
With appropriate choices of $\psi$, we obtain $\Delta^{\log} = \Delta {\log}$ and $\Delta^{\sqrt{\cdot}}(u)=\frac {\Gamma(\sqrt u)} u - \frac {\Delta u}{2u}$ for all positive functions $u$ (cf. Example \ref{Elog}). Especially with these two choices, the $\psi$-Li-inequality (\ref{psiLY}) turns into the logarithmic Li-Yau inequality (\ref{logLY}), respectively into the $\sqrt{\cdot}$ Li-Yau inequality (\ref{BLY}).

We will introduce a first $\psi$-gradient in Subsection \ref{SGF}. We will need this to prove Harnack inequalities.
In the manifolds subsection (Subsection \ref{SMF}) of this section, we will show among others that, in the case of manifolds for suitable $\psi$, one has

\begin{eqnarray*}
\Delta^\psi &=& \Delta \log (\cdot),\\
\Gamma^\psi &=& \Gamma (\log (\cdot)),\\
\Gamma_2^\psi &=& \Gamma_2 (\log (\cdot)).
\end{eqnarray*}

In the limit theorem subsection (Subsection \ref{SLT}), we will describe the classical operators as a limit of the $\psi$-operators. 
I.e. for all $f \in C(V)$ and all concave $\psi \in C^\infty(\posR)$, one has
\begin{eqnarray*}
\lim_{\varepsilon \to 0} \frac 1 {\varepsilon} \Delta^\psi (1+ \varepsilon f) &=& \psi'(1) \Delta f, \\
\lim_{\varepsilon \to 0} \frac 1 {\varepsilon^2} \Gamma^\psi (1+ \varepsilon f) &=& -\psi''(1) \Gamma(f), \\
\lim_{\varepsilon \to 0} \frac 1 {\varepsilon^2} \Gamma_2^\psi (1+ \varepsilon f) &=& -\psi''(1) \Gamma_2(f). 
\end{eqnarray*}
By this, we show that the $CD\psi$ inequality implies the $CD$ inequality.

We will prove below our key identity
\[
\LL(u \Delta^\psi u)) = -2u\Gamma_2^\psi (u)
\]
which holds for all solutions  $\sol$  to the heat equation.


In the examples section (Section \ref{CExamples}), we will show that Ricci-flat graphs satisfy the $CD\psi(d,0)$ inequality with appropriate $\psi \in C^1(\posR)$ and $d \in \posR$.

We will see that concavity of the function $\psi$ is crucial for deriving Harnack inequalities (cf. Section \ref{CHarnack}) and for proving $CD\psi$ inequalities on Ricci-flat graphs (cf. Subsection \ref{SRF}).

  \subsection{The operator $\Delta^\psi$} \label{SDpsi} 

In this subsection, we will introduce the scaling invariant Laplacian $\Delta^\psi$ as replacement of $\Delta \circ \log$ which appeared in the manifold case \ref{LY86}. We will show in Subsection \ref{SMF} that both coincide in the case of manifolds if $\psi'(1)=1=-\psi''(1)$, as holds e.g. for $\psi = \log$.


\begin{defn}[$\psi$-Laplacian $\Delta^\psi$] \label{DLP}
Let $\psi \in C^1(\posR)$ and let $G=(V,E)$ be a finite graph. Then, we call
$\Delta^\psi : C^+(V) \to C(V)$, defined as 
\[
(\Delta^\psi f ) (v) :=  \left( \Delta \left[ \psi \left( \frac f {f(v)} \right) \right] \right) (v),
\]
the $\psi$\emph{-Laplacian}. 
\end{defn}

We show that the operator $\Delta^\psi$ is scaling invariant in the argument and linear in the parameter $\psi$.
In particular, $\Delta^\psi$ is not a linear operator.
\begin{lemma} \label{Lcalc}
Let $G=(V,E)$ be a graph, $\phi, \psi \in C^1(\posR)$, $a,b \in \IR$, $r \in \posR$, and $f \in C^+(V)$. Then,
\eBr
	\item $\Delta^{a \phi + b \psi} (rf) = a \Delta^\phi f + b \Delta^\psi f$, 
	\item $\Delta^\psi f = 0$ if $f$ is constant or $\psi$  is constant.
\end{enumerate}
\end{lemma}
\begin{proof}
The first claim is a direct consequence of the definition of $\Delta^\psi$ and the linearity of $\Delta$. The second claim follows from the fact that the Laplacian vanishes on constant functions.
\end{proof}

With appropriate choices of $\psi$, the $\psi$-Li-Yau inequality (\ref{psiLY}) turns into the logarithmic Li-Yau inequality (\ref{logLY}) respectively into the $\sqrt{\cdot}$ Li-Yau inequality (\ref{BLY}). This is discussed next.

\begin{example} \label{Elog}
Let $G=(V,E)$ be a finite graph.

	(1) 
		If $\psi = \log$, and $f \in C(V)$, and $v \in V$, then we get
		\begin{eqnarray*}
			\Delta^{\log}f(v) &=& \Delta \left[ \log \left(\frac f {f(v)}\right)\right](v) = (\Delta \log f) (v) - \Delta (\log f(v)) \\ &=& (\Delta \log f) (v)
		\end{eqnarray*}
		since $f(v)$ is a constant and the Laplacian vanishes on constants. Hence, we have
		\[
			-\Delta^{\log}f =  -\Delta \log f.
		\]
		This means, the $\psi$-Li-Yau inequality (\ref{psiLY}) turns into the logarithmic Li-Yau inequality (\ref{logLY}).

	(2)	
	  If $\psi = \sqrt{\cdot}$ and $\sol$ is a positive solution to the heat equation on $G$ and $v \in V$, then
		\[
			\Delta^{\psi}u(v) = \Delta  \sqrt{\frac u {u(v)}}  (v) = \frac   { (\Delta \sqrt{u})(v)} {\sqrt{u(v)}}
		\] 
		since the Laplacian is linear. Thus,
		\begin{eqnarray*}
			-\Delta^{\psi}u &=&  -\frac {\Delta  \sqrt u} {\sqrt u} \\
										  &=&  \frac {\Delta \left[{\sqrt u}^2\right] - 2 \sqrt u \Delta \sqrt{u} } {2u} -  \frac {\Delta u} {2 u} \\
										  &=&  											\frac {\Gamma(\sqrt u)} u -  \frac {\Delta u} {2 u} \\
											&\stackrel{\LL(u) =0}{=}& \frac {\Gamma(\sqrt u)} u -  \frac {\partial_t u}{2u} \\
											&=& \frac {\Gamma(\sqrt u)} u - \frac {\partial_t \sqrt {u}}{\sqrt{u}}.
		\end{eqnarray*}
		This means, the $\psi$-Li-Yau inequality (\ref{psiLY}) turns into the $\sqrt{\cdot}$ Li-Yau inequality (\ref{BLY}).

\end{example}

\begin{rem}
In the Example \ref{Elog}.(2), we have seen the chain rule for the square root on graphs which is a key identity in \cite{Bauer2013}. 
This chain rule states that on finite graphs $G=(V,E)$ for all $f \in C^+(V)$
\[
2 \sqrt{f}\Delta \sqrt{f} =  \Delta f - 2 \Gamma (\sqrt{f}).
\]
\end{rem}

There is a useful representation of $\Delta^\psi$.

\begin{lemma} [Representation of $\Delta^\psi$] \label{LRL}
Let $\psi \in C^1(\posR)$.
 
    Let $G=(V,E)$ be a finite graph, let $f \in C^+(V)$ and let $v \in V$. If  $\psi(1)=0$, then
    \begin{equation}
      \Delta^\psi f (v) = \sum_{w \sim v} \psi \left( \frac {f(w)}{f(v)} \right). \label{repG}
    \end{equation}
\end{lemma}

\begin{proof}
This claim is obvious, since $0 = \psi(1) = \psi \left( \frac {f(v)}{f(v)} \right)$ .
\end{proof}

  \subsection{The $CD\psi$ condition and $\Gamma_2^\psi$} \label{SCD}
  
In this subsection, we introduce the non bilinear operator $\Gamma_2^\psi$ and formulate the $CD\psi$ condition. 
 Furthermore, we will prove our key identity $\LL(u \Delta^\psi u)) = -2 u\Gamma_2^\psi (u)$ whenever $\LL u  = 0$. We will use this key identity to characterize the $CD\psi$ condition.

\begin{defn}[Second $\psi$-gradient $\Gamma_2^\psi$] \label{DefG2} 
Let $\psi \in C^1(\posR)$, and let $G=(V,E)$ be a finite graph. Then, we define
$\Omega^\psi : C^+(V) \to C(V)$ by 
\[
(\Omega^\psi f ) (v) :=  \left( \Delta \left[ \psi' \left( \frac f {f(v)} \right)  \cdot \frac f {f(v)} \left[ \frac{\Delta f} {f} - \frac{(\Delta f)(v)} {f(v)} \right]  \right] \right) (v).
\]

Furthermore, we define the \emph{second} $\psi$-\emph{gradient} $\Gamma_2^\psi : C^+(V) \to C(V)$ by
\[
2 \Gamma_2^\psi (f) := \Omega^\psi f + \frac {\Delta f \Delta^\psi f} f - \frac {\Delta \left(f \Delta^\psi f\right)} f.
\]
\end{defn}

In the next subsection, we will also define the first $\psi$-gradient $\Gamma^\psi$. But there is no obvious derivation of $\Gamma_2^\psi$ from $\Gamma^\psi$ in a similar way as $\Gamma_2$ derives from $\Gamma$ (see Definition \ref{defGamma}). Instead, one has a variant of our key identity, namely
\[
\LL\left(u \Gamma^\psi u \right) =  2 u \Gamma_2^\psi u , \quad \mbox{if } \LL u =0.
\]
Moreover, in the case of manifolds, one has $\Gamma_2^\psi = \Gamma_2 (\log (\cdot))$ for suitable $\psi$.
This result will be discussed in the manifolds subsection (Subsection \ref{SMF})

With the previous definition, we have all ingredients for defining the $CD\psi$ condition.

\begin{defn}[$CD\psi$ condition]
Let $G=(V,E)$ be a finite graph and $d \in \posR$. We say $G$ satisfies the $CD\psi(d,0)$ \emph{inequality} if for all $f \in C^+(V)$, one has 
\begin{equation}
  \Gamma_2^\psi (f) \geq \frac 1 d (\Delta^\psi f)^2. \label{CDpsi}
\end{equation}
\end{defn}


Now, we will give the key identity of the second $\psi$-gradient $\Gamma_2^\psi$. 
It will be used to show the analogy between the classical Bakry-Emery calculus and the $\Gamma^\psi$-calculus, and to characterize the validity of the $CD\psi$ inequality.

\begin{lemma} [Representation of $\Gamma_2^\psi$]
Let $\psi \in C^1(\posR)$, let $I \subset \IR$ be an interval, and let $t_0 \in I$.
 
      Let $G=(V,E)$ be a finite graph, $u : C^1(V \times I)$ a positive function with $\LL(u) (\cdot,t_0)=0$ . Then
      \begin{equation}
        \LL \left( - u \Delta^\psi u \right) (\cdot,t_0) = 2 u \Gamma_2^\psi (u) (\cdot,t_0). \label{repGG}
      \end{equation}
      
\end{lemma}

\begin{proof}
All computations of the proof of the first claim are taking place at the time $t=t_0$.
\begin{eqnarray*}
\LL \left(-u \Delta^\psi u\right) &=& - \Delta \left(u \Delta^\psi u \right) + \partial_t \left(u \Delta^\psi u \right) \\
&=& - \Delta \left(u \Delta^\psi u \right) + \left( \Delta u \right)  \Delta^\psi u  + u  \partial_t  \Delta^\psi u .\\
\end{eqnarray*}
For all $v \in V$ we compute 
\begin{eqnarray*}
\partial_t  \Delta^\psi u (v) &=& \Delta \partial_t \left[ \psi\left(\frac u {u(v)}\right)\right](v) \\
&=& \Delta \left[ \psi'\left(\frac u {u(v)}\right) \cdot \partial_t \left(\frac u {u(v)}\right)   \right](v) \\
&=& \Delta \left[ \psi'\left(\frac u {u(v)}\right) \cdot \frac{u(v)\partial_t u - u \partial_t u(v)}{u(v)^2}   \right](v) \\
&=& \Delta \left[ \psi'\left(\frac u {u(v)}\right) \cdot \frac u {u(v)} \cdot \left[\frac{\Delta u} {u} - \frac{\Delta u(v)} {u(v)}  \right]   \right](v) \\
&=& \Omega^\psi u (v).
\end{eqnarray*}

Putting the two equalities above together results for all $v \in V$ in
\[
\LL \left(-u \Delta^\psi u\right) = - \Delta \left(u \Delta^\psi u \right) + \left(\Delta u \right) \Delta^\psi u  + u \Omega^\psi u = 2u\Gamma_2^\psi (u).
\]
This finishes the proof.
\end{proof}

Next, we present a characterization of the $CD\psi$ condition which is the key to prove the $\psi$-Li-Yau inequality (\ref{psiLY}).

\begin{theorem} [Characterization $CD\psi$] \label{tchar}
Let $G=(V,E)$ be a finite graph, $\psi \in C^1(\posR)$, and $d \in \posR$. Then, the following statements are equivalent.
\eChar
	\item   \label{1}
	  $G$ satisfies the $CD\psi(d,0)$ inequality. 
	\item   \label{2}
	  For all  positive solutions $\sol$ to the heat equation on $G$, one has
	  \begin{equation}
	    \LL \left( -u \Delta^\psi u \right) \geq \frac 2 d u \left( \Delta^\psi u \right)^2. \label{CCD}
	  \end{equation}
\end{enumerate}
\end{theorem}

This characterization shows the connection between the time-independent $CD\psi$ condition and its consequences for solutions to the heat equation.

\begin{proof}

The implication $\emph{(\ref{1})} \Rightarrow \emph{(\ref{2})}$ is a direct consequence of the identity (\ref{repGG}) from the representation lemma of $\Gamma_2^\psi$, i.e.
\[
  \LL \left( -u \Delta^\psi u \right) \stackrel{(\ref{repGG})}{=} 2 u \Gamma_2^\psi (u) \stackrel{\emph{(\ref{1})}}{\geq} \frac 2 d u \left( \Delta^\psi u \right)^2.
\]

Next, we show $\emph{(\ref{2})} \Rightarrow \emph{(\ref{1})}$.

Given $f \in C^+(V)$, there is a positive solution $\sol$ to the heat equation on $G$, such that
$u(\cdot,0) = f$.
Such a solution $u$ can be obtained by
\[
u_t := e^{\Delta t} f := \sum_{k=0}^{\infty} \frac {t^k \Delta^k f}{k!}
\]
for $x \in V$ and  $t \geq 0$.  This statement will be proved in Proposition \ref{SG}. We use the continuity of $u$ and $\Delta^\psi$ and $\Gamma_2^\psi$,  and we apply (\ref{repGG}) to get
\begin{eqnarray*}
2f \Gamma_2^\psi (f) 		&=& 							\lim_{t \downto 0} 2u \Gamma_2^\psi (u)(\cdot,t) 
\stackrel{(\ref{repGG})}{=} 						\lim_{t \downto 0}  \LL \left( -u \Delta^\psi u \right) (\cdot,t)   \\
&\stackrel{\emph{(\ref{2})}}{\geq}& 						\lim_{t \downto 0} \frac 2 d u \left( \Delta^\psi u \right)^2 (\cdot,t)
 = 																			\frac 2 d f \left( \Delta^\psi f \right)^2.
\end{eqnarray*}
This finishes the proof.
\end{proof}

\subsection{The gradient form $\Gamma^\psi$} \label{SGF}

In this subsection, we introduce the gradient form $\Gamma^\psi$. This is not necessary to understand the $CD \psi$ condition and the $\psi$-Li-Yau inequality.
But it enables us to formulate the $\psi$-Li-Yau inequality as a gradient estimate. This formulation is crucial to derive Harnack inequalities.

\begin{defn}
For $\psi \in C^1(\posR)$, we define
\[
\overline{\psi}(x):= \psi'(1)\cdot(x-1)  - (\psi(x) - \psi(1)). 
\]
\end{defn}

If $\psi$ is a concave function, then obviously, one has $\overline{\psi}(x) \geq 0$ for all $x>0$.
We can see $\overline{\psi}(1)=0$.
Hence, if $\psi$ is concave and if $G=(V,E)$ is a finite graph, then for every $v_0 \in V$ and every positive $f \in C^+(V)$, the function $V \to \IR$ with
\[
v \mapsto \overline{\psi} \left( \frac {f(v)}{f(v_0)} \right)
\]
has a minimum in $v=v_0$.
Consequently
\[
\Delta^{\overline{\psi}} f(v_0) = \Delta \overline{\psi} \left( \frac {f}{f(v_0)} \right) \geq 0.
\]

This motivates the following definition.

\begin{defn}[$\psi$-gradient $\Gamma^\psi$] 
 Let $\psi \in C^1(\posR)$ be a concave function and let $G=(V,E)$ be a finite graph. 
We define the $\psi$-\emph{gradient} as $\Gamma^\psi : C^+(V) \to C(V)$,
\[
\Gamma^\psi  := \Delta^{\overline{\psi}}.
\]
\end{defn}

\begin{rem}
By using this, we can also introduce a $CD\psi$ inequality for a dimension $d>0$ with a non-zero curvature bound $K \in \IR$ denoted by $CD\psi(d,K)$ via the inequality
\[
\Gamma_2^\psi( f) \geq \frac 1 d \left(\Delta^\psi f \right)^2 + K\Gamma^\psi( f)  \quad \mbox{for all } f \in C^+(V).
\]
In this paper, we focus on the case $K=0$.
\end{rem}

In the following lemma, we give a representation of $\Delta^\psi$ which allows us to understand the $\psi$-Li-Yau inequality as a gradient estimate.
This lemma will be used in Section \ref{CHarnack} to derive Harnack inequalities.
\begin{lemma}[Gradient representation of $\Delta^\psi$] \label{LGR} 
Let $G=(V,E)$ be a finite graph and let $\psi \in C^1(\IR)$ be a concave function. Then for all $f \in C^+(V)$, one has
\begin{equation}
-\Delta^\psi f = \Gamma^\psi (f) - \psi'(1) \frac {\Delta f} f. \label{EGamma}
\end{equation}
\end{lemma}
\begin{proof}
The claim follows immediately from Lemma \ref{Lcalc} and from the definition of $\Delta^\psi$ (Definition \ref{DLP}).
\end{proof}

\subsection{The $\Gamma^\psi$-calculus on manifolds} \label{SMF} 

In this subsection, we will see that the new $\Gamma^\psi$ calculus can be transferred to the setting of manifolds.
Our goal is to show that in this setting, the $\Gamma^\psi$ calculus coincides with the common $\Gamma$-calculus by Bakry and Émery (cf. \cite{Bakry1985,Bakry1992,Bakry2006}).
For basics on Riemannian manifolds, we refer the reader to \cite{Spivak1965}.

\begin{defn}[$\psi$-Laplacian on manifolds] 
Let $\psi \in C^1(\posR)$ and let $\LB$ be the Laplace-Beltrami operator on a complete, connected Riemannian manifold $M$. Then, we call
$\Delta^\psi : C_+^\infty(M) \to C^\infty(M)$ with 
\[
(\Delta^\psi f ) (v) :=  \left( \Delta \left[ \psi \left( \frac f {f(v)} \right) \right] \right) (v)
\]
the $\psi$\emph{-Laplacian}. 
\end{defn}

Analogously, we can transfer the definitions from the previous three subsections to the setting of manifolds.
By doing this transfer, the equations (\ref{repGG}) and (\ref{EGamma}) remain valid also on manifolds. The proofs are analogous to the ones in the graph case.

But in contrast to the graph case, the $\Gamma^\psi$-calculus on manifolds only depends on $\psi'(1)$ and $\psi''(1)$. We will discuss this in the following representation theorem.

\begin{theorem} [Representation of the $\psi$-operators] \label{TPO}
Let $\psi \in C^1(\posR)$. Let $\LB$ be the Laplace-Beltrami operator on a complete, connected Riemannian manifold $M$ and $\psi \in C^\infty(\posR)$. Then for all positive $f \in C^\infty(M)$, one has
     \begin{eqnarray}
       \LB^\psi f    &=& \psi'(1)  \frac{\LB f}{f} + \psi''(1) \frac{\Gamma  (f)}{f^2}, \label{TMP}\\ 
       \Gamma^\psi   &=& -\psi''(1)  \Gamma (\log (\cdot)), \label{clGa} \\
       \Gamma_2^\psi &=& -\psi''(1)  \Gamma_2 (\log (\cdot)). \label{TM2}    
    \end{eqnarray}    
\end{theorem}

\begin{proof}
Let $f \in C^\infty(M)$ and $x \in M$.

First, we prove identity (\ref{TMP}).

We use the chain rule (\ref{CR}) for
\[
\Phi(s) = \psi \left( \frac s {f(x)} \right),  \qquad s>0 
\]
to get
\begin{eqnarray*}
\LB^\psi f (x) &=& (\LB \Phi(f))(x) \stackrel{(\ref{CR})}{=} \Phi'(f(x)) (\LB f) (x) + \Phi''(f(x)) (\Gamma (f)) (x) \\
&=&  \frac {\psi'(1)}{f(x)} (\LB f) (x) + \frac{\psi''(1)}{(f(x))^2} (\Gamma (f)) (x) \\
&=&  \left[ \psi'(1) \frac {\LB f}{f} + \psi''(1) \frac{\Gamma (f)}{f^2} \right] (x).
\end{eqnarray*}

Since $f$ and $x$ are arbitrary, the claim follows immediately.

Next, we show identity (\ref{clGa}).

Since equation (\ref{EGamma}) remains valid on manifolds, we can calculate 
\[
  \Gamma^\psi (f) \stackrel{(\ref{EGamma})}{=} \psi'(1) \frac {\LB f} f -\LB^\psi f =  - \psi''(1) \frac{\Gamma  (f)}{f^2} \stackrel{(\ref{CRG})}{=} -\psi''(1)  \Gamma (\log  f)
\]
by using the already proven first statement of the present theorem and the chain rule (\ref{CRG}).


Finally, we prove identity (\ref{TM2}) in two steps.
In the first step, we show $\Gamma_2^{\log} (f) = \Gamma_2 (\log f)$, and in the second step, we show $\Gamma_2^\psi = -\psi''(1)\Gamma_2^{\log}$.

We start with $\Gamma_2^{\log} (f) = \Gamma_2 (\log f)$.

For all $x \in M$, we obtain

\begin{eqnarray*}
(\Omega^{\log} f ) (x) &=&    \left( \LB \left[ {\left( \frac f {f(x)} \right)}^{-1}  \cdot \frac f {f(x)} \left[ \frac{\LB f} {f} - \frac{(\LB f)(x)} {f(x)} \right]  \right] \right) (x) \\
															&=&     \left( \LB \left[  \frac{\LB f} {f} - \frac{(\LB f)(x)} {f(x)}   \right] \right) (x) \\
															&=&     \left( \LB \left[  \frac{\LB f} {f}   \right] \right) (x). \\
\end{eqnarray*}
This implies
\[
\Omega^{\log} f = \LB \left(\frac{\LB f}{f} \right).
\]

Now, we use the chain rules
\begin{eqnarray}
\Gamma(\log f,g) &=& \frac {\Gamma(f,g)} f, 											\label{crl1} \\ 
\LB \log f    &=& \frac {\LB f} f - \frac {\Gamma(f)} {f^2}   \label{crl2}
\end{eqnarray}
 to obtain
\begin{eqnarray*}
			2\Gamma_2(\log f) 
&=&															\LB \Gamma (\log f)			- 2 \Gamma (\log f, \LB \log f) \\
&\stackrel{(\ref{crl1})}{=}&			\LB 	\left( \frac{\Gamma(f)}{f^2} \right) - \frac 2 f \Gamma(f,\LB \log f) \\
&=&							  							\LB \left( \frac{\Gamma(f)}{f^2} \right) - \frac 1 f [\LB(f \LB \log f) - f \LB \LB \log f - \LB f \LB \log f] \\
&=&													    \LB \left[ \LB \log f + \frac{\Gamma(f)}{f^2} \right]- \frac{\LB(f \LB \log f)} f + \frac {\LB f \LB \log f} f \\
&\stackrel{(\ref{crl2})}{=}&		  \LB \left( \frac{\LB f} {f} \right)- \frac{\LB(f \LB \log f)} f + \frac {\LB f \LB \log f} f \\
&=&								 						  \Omega^{\log} f - \frac{\LB(f \LB \log f)} f + \frac {\LB f \LB \log f} f \\
&=&								 						  2 \Gamma_2^{\log} (f).
\end{eqnarray*}

Before we prove the second step, we establish the equation
\begin{equation}
\LL\left(u \Gamma^\psi u  \right)=\LL\left( -u \LB^\psi u \right) = 2u\Gamma_2^\psi u \label{ErepG}
\end{equation}
for all positive solutions $u$ to the heat equation on $M \times \nnegR$.

The proof of the second identity is analogous to the proof of the key identity (\ref{repGG}) on graphs.
For showing the first identity, we write

\begin{eqnarray*}
\LL \left( -u \LB^\psi u \right) &=& \LL\left[ u \left( \Gamma^\psi (u) - \psi'(1) \frac {\Delta u} u \right) \right] \\
															   &=& \LL\left(u \Gamma^\psi u  \right)- \psi'(1) \LL(\Delta u)	  \\
															   &=& \LL\left(u \Gamma^\psi u  \right). 
\end{eqnarray*}

We prove $\Gamma_2^\psi (f) = \Gamma_2^{\log} (f)$.

We will use the already proven second claim (\ref{clGa}) of the present theorem, and we will use the identity (\ref{ErepG}) from above.
Therefore, we extend the positive function $f \in C^\infty(M)$ to a positive function $u \in C^\infty(M \times \IR)$, such that $u(\cdot,0) = f$ and $\LL(u) (\cdot,0) = 0$.
We can do this by 
\[
u(\cdot,t) := f \cdot \exp \left( \frac{\LB f}{f} t \right).
\]
At the time $t=0$, we calculate 
\begin{eqnarray*}
2 f \Gamma_2^\psi (f) &= &																2 u \Gamma_2^\psi (u) 
										\stackrel{(\ref{ErepG})}  {=}  		           \LL \left( u \Gamma^\psi  u \right) 
										\stackrel{(\ref{clGa})}   {=} 		-\psi''(1) \cdot \LL \left( u \Gamma^{\log} u \right) \\
										&\stackrel{(\ref{ErepG})}  {=}&			-\psi''(1) \cdot 2 u \Gamma_2^{\log} (u) 
										=  -\psi''(1)\cdot 2 f \Gamma_2^{\log} (f).
\end{eqnarray*}
This finishes the proof, since $f$ is an arbitrary positive function.
\end{proof}

\begin{rem}
In the proof of \cite[Lemma 3.12]{Bauer2013}, similar computations can be found to show the connection between their $CDE$ condition and the $CD$ condition.
\end{rem}

With appropriate $\psi$, the representation of the $\psi$-operators simplifies to the following corollary.

\begin{corollary}[Coincidence of the $\psi$-operators with the $\Gamma$ calculus]
Let $\psi \in C^\infty$ with $\psi'(1) = 1 = - \psi''(1)$, and let $\LB$ be the Laplace-Beltrami operator on a complete, connected Riemannian manifold $M$. Then for all positive $f \in C^\infty(M)$, one has
    \begin{eqnarray}
       \LB^\psi f 				\;=&		\LB^{\log} f 				&=\;   \LB \log f, \label{Apsi} \\
       \Gamma^\psi (f) 		\;=& 	\Gamma^{\log} (f) 	  &=\; 	\Gamma (\log f), \\
       \Gamma_2^\psi (f) 	\;=&		\Gamma_2^{\log} (f) &=\; 	\Gamma_2 (\log f).
    \end{eqnarray}      
\end{corollary}

This corollary shows that our replacement $\Delta \log (\cdot) \leadsto  \Delta^\psi$ is consistent.
Moreover, we can see the equivalence of the $CD$ condition and the $CD \psi$ condition on manifolds.

\begin{proof}
If we choose $\psi=\log$, then we obtain $\psi'(1) = 1 = - \psi''(1)$.
Thus, the claim is an easy consequence of the representation of the $\psi$-operators (Theorem \ref{TPO}).
\end{proof}


  \subsection{A limit theorem on graphs} \label{SLT}  

In this subsection, we will interpret the classical operators $\Delta, \Gamma, \Gamma_2$ as directional derivatives of the corresponding $\psi$-operators.
By this, we can show that the $CD\psi$ condition implies the $CD$ condition on graphs.

\begin{theorem} [Limit of the $\psi$-operators]
Let $G=(V,E)$ be a finite graph. Then for all $f \in C(V)$, one has the pointwise limits 
\begin{eqnarray}
\lim_{\varepsilon \to 0} \frac 1 {\varepsilon} \Delta^\psi (1+ \varepsilon f)     \; =& \psi'(1) \Delta f 			   &\quad \mbox{for } \psi \in C^1(\posR), \label{ELL}\\
\lim_{\varepsilon \to 0} \frac 1 {\varepsilon^2} \Gamma^\psi (1+ \varepsilon f)   \; =& -\psi''(1) \Gamma(f) 	   	 &\quad \mbox{for } \psi \in C^2(\posR), \label{EL1}\\
\lim_{\varepsilon \to 0} \frac 1 {\varepsilon^2} \Gamma_2^\psi (1+ \varepsilon f) \; =& -\psi''(1) \Gamma_2(f) 		 &\quad \mbox{for } \psi \in C^2(\posR). \label{EL2}
\end{eqnarray}
Since all $f \in C(V)$ are bounded, one obviously has $1+ \varepsilon f> 0$ for small enough $\varepsilon>0$.
\end{theorem}

\begin{rem}
The above limits can be understood as directional derivatives of the $\psi$-operators to the direction $f$ at the constant function $1$.
For a function $H : C(V) \to C(V)$ and for $f,g \in C(V)$ and for $x \in V$, we can define $h : \IR \to \IR, \; t \mapsto H(g + tf)(x)$ and the directional derivatives
$\partial_f H, \partial^{2}_f H: C(V) \to C(V)$ via
$ \big[\partial_f H (g)\big](x) = h'(0)$ and $ \big[\partial_f^2 H (g)\big](x) = h''(0)$. By using this notation, for $\psi \in C^2(\posR)$ with $\psi'(1)\neq 0 \neq \psi''(1)$, the above theorem can be written as
\begin{eqnarray*}
\Delta f    &=& \left[ \partial_f \Delta^\psi \right] (1) / \psi'(1) 														  	,\\
\Gamma(f)   &=& -\frac 1 2 \left[  \partial^2_f \Gamma^\psi \right] (1) / \psi''(1) 									,\\
\Gamma_2(f) &=& -\frac 1 2 \left[  \partial^2_f \Gamma_2^\psi \right] (1) / \psi''(1) 								.
\end{eqnarray*}
\end{rem}

\begin{proof}
First, we prove (\ref{EL1}). Let $x \in V$, let $\psi \in C^2(\posR)$ and let $f \in C(V)$.  By the definition of $\Gamma^\psi$, we can write
\[
\frac 1 {\varepsilon^2} \Gamma^{\psi}(1+ \varepsilon f)(x) = \Delta \frac 1 {\varepsilon^2} \overline{\psi} \left( \varepsilon \cdot \frac  {f-f(x)}{1+\varepsilon f(x)} + 1 \right)(x).
\] 
We recall $\overline{\psi}(t) =  \psi'(1) (t-1) - \psi(t) $. W.l.o.g., $\psi(1)=0$ and, hence, $\overline{\psi}(1)=0=\overline{\psi}'(1)$ and $\overline{\psi}''(1)=-\psi''(1)$.
We apply Taylor's theorem to $\overline{\psi}$ at point $1$ to obtain
\begin{equation*}
\overline{\psi} \left( \varepsilon \cdot \frac  {f-f(x)}{1+\varepsilon f(x)} + 1 \right) = \left(\frac 1 2 \overline{\psi}''(1)+ h(t_\varepsilon)  \right) t_\varepsilon^2 
\end{equation*}
with $t_\varepsilon = \varepsilon \cdot \frac  {f-f(x)}{1+\varepsilon f(x)}$ and a function $h: \IR \to \IR$ with $\lim_{s\to 0} h(s) = 0$.
Thus, we can calculate the pointwise limit
\[
\lim_{\varepsilon \to 0} \frac 1 {\varepsilon^2} \overline{\psi} \left( \varepsilon \cdot \frac  {f-f(x)}{1+\varepsilon f(x)} + 1 \right) = \frac 1 2 \overline{\psi}''(1)   (f-f(x))^2.
\]
Since we are on finite graphs, we can  do the following computation at the point $x \in V$:
\begin{eqnarray*}
\lim_{\varepsilon \to 0} \frac 1 {\varepsilon^2} \Gamma^\psi (1+ \varepsilon f)  
&=& \Delta \lim_{\varepsilon \to 0} \frac 1 {\varepsilon^2} \overline{\psi} \left( \varepsilon \cdot \frac  {f-f(x)}{1+\varepsilon f(x)} + 1 \right)\\
&=& \Delta  \frac 1 2 \overline{\psi}''(1)   (f-f(x))^2\\
&=& -\psi''(1) \Gamma(f).
\end{eqnarray*}

Next, we prove (\ref{ELL}) similarly. Let $f \in C(V)$ and let $\psi \in C^1(\posR)$. 
Since $\overline{\psi}(1)=0$, we observe that 
\[
\frac 1 \varepsilon  \Gamma^\psi (1+ \varepsilon f) (x) = \Delta \frac 1 \varepsilon  \overline{\psi} \left(\frac{1+ \varepsilon f}{1+ \varepsilon f(x)}\right) (x) = \Delta \frac 1 \varepsilon \left[ \overline{\psi} \left( \varepsilon \cdot \frac  {f-f(x)}{1+\varepsilon f(x)} + 1 \right) -\overline{\psi}(1) \right](x)
\]
for all $x \in V$. Due to Taylor's theorem of the first order and since $\overline{\psi}'(1)$=0, we can compute the pointwise limit
\[
\lim_{\varepsilon \to 0} \frac 1 {\varepsilon} \overline{\psi} \left( \varepsilon \cdot \frac  {f-f(x)}{1+\varepsilon f(x)} + 1 \right) =  \overline{\psi}'(1)   (f-f(x)) = 0.
\]
Hence, we have $\frac 1 \varepsilon  \Gamma^\psi (1+ \varepsilon f) \stackrel{\varepsilon \to 0}{\longrightarrow} 0$.
We use the gradient representation of $\Delta^\psi$ (cf. Lemma~\ref{LGR}) to obtain
\begin{eqnarray*}
\frac 1 \varepsilon \Delta^\psi (1+ \varepsilon f) &=& \frac 1 \varepsilon  \psi'(1) \frac {\Delta (1+ \varepsilon f)}{1+ \varepsilon f} - \frac 1 \varepsilon  \Gamma^\psi (1+ \varepsilon f) \\
&=&  \psi'(1) \frac {\Delta  f}{1+ \varepsilon f} - \frac 1 \varepsilon  \Gamma^\psi (1+ \varepsilon f)\\
&\stackrel{\varepsilon \to 0}{\longrightarrow}& \psi'(1)  {\Delta  f}.
\end{eqnarray*}

Finally, we prove (\ref{EL2}).
Let $x \in V$, let $\psi \in C^2(\posR)$ and let $g \in C^+(V)$. Again, we use the gradient representation of $\Delta^\psi$ (cf. Lemma \ref{LGR}) and calculate
\begin{equation}
\Delta(g \Delta^\psi g) = \Delta (\psi'(1) \Delta g) - \Delta \left( g \Gamma^\psi (g) \right). \label{LgGg}
\end{equation}
We recall the non-linear operator $\Omega^\psi$ from the definition of $\Gamma_2^\psi$  (cf. Definition \ref{DefG2}). We can write
\begin{eqnarray}
\left[g\Omega^\psi(g)\right](x) 
&=& \left[  g \Delta \left[ \psi'\left(\frac g {g(x)}\right) \frac g {g(x)} \cdot \left( \frac {\Delta g} g - \frac {\Delta g (x)}{g(x)} \right) \right] \right](x) \nonumber\\
&=& \left[ \Delta \left[ \psi'\left(\frac g {g(x)}\right) \Delta g \right] - \Delta g \cdot \Delta \left[ \psi'\left(\frac g {g(x)}\right) \frac g {g(x)} \right] \right](x). \label{gOg}
\end{eqnarray}
We define $\nu(t)= \psi'(t)-\psi'(1)$ and $\omega(t)= \psi(t)-t\psi'(t)$ for $t \in \posR$. By using the linearity of $\Delta^\psi$ in the parameter and by resolving $g\Omega^\psi(g)$ with 
(\ref{gOg}) and $\Delta(g \Delta^\psi g)$ with (\ref{LgGg}), we obtain
\begin{eqnarray*}
		\left[    2g \Gamma_2^\psi(g) \right](x)  
&=& \left[    g\Omega^\psi(g) + \Delta g \Delta^\psi g - \Delta(g \Delta^\psi g) \right](x) \\
&=& \left[    \Delta\left[\nu \left( \frac g {g(x)} \right) \Delta g \right] + \Delta g \Delta^\omega g + \Delta \left( g \Gamma^\psi (g) \right) \right](x).
\end{eqnarray*}
Let $f \in C(V)$. We set $g = g_\varepsilon = 1+\varepsilon f$ and get
\[
\left[   
\frac {2g_\varepsilon} {\varepsilon^2}  \Gamma_2^\psi(1+ \varepsilon f)
\right](x)       
= 
\left[
\frac 1 {\varepsilon^2} \Delta\left[\nu \left( \frac {g_\varepsilon} {g_\varepsilon(x)} \right) \Delta g_\varepsilon \right] + 
\frac 1 {\varepsilon^2} \Delta g_\varepsilon \Delta^\omega g_\varepsilon + 
\frac 1 {\varepsilon^2} \Delta \left( g_\varepsilon \Gamma^\psi (g_\varepsilon) \right)
\right](x)  .
\]
Now, we compute the limits of the three summands on the right hand side. We start with the last one and proceed backwards.

Since $g_\varepsilon \stackrel{\varepsilon \to 0}{\longrightarrow} 1$ and 
$\frac 1 {\varepsilon^2}  \Gamma^\psi (g_\varepsilon) \stackrel{\varepsilon \to 0}{\longrightarrow} -\psi''(1) \Gamma(f)$ which was proven above, one has
\[
\frac 1 {\varepsilon^2} \Delta \left( g_\varepsilon \Gamma^\psi (g_\varepsilon) \right)   \stackrel{\varepsilon \to 0}{\longrightarrow} -\psi''(1) \Delta  \Gamma (f).
\] 

Since $\frac 1{\varepsilon} \Delta  g_\varepsilon = \Delta f$ and 
$\frac 1 {\varepsilon}  \Delta^\omega (g_\varepsilon) \stackrel{\varepsilon \to 0}{\longrightarrow} \omega'(1) \Delta f  = -\psi''(1) \Delta f$ which was proven above, one has
\[
\frac 1 {\varepsilon^2} \Delta g_\varepsilon \Delta^\omega g_\varepsilon  \stackrel{\varepsilon \to 0}{\longrightarrow} -\psi''(1) (\Delta f)^2.
\] 

The function $\nu$ is differentiable since $\psi \in C^2(\posR)$, and obviously, $\nu(1)=0$. By defining $\delta := \varepsilon \cdot \frac {f-f(x)}{g_\varepsilon}$, we obtain
\[
\frac 1 \varepsilon \nu \left( \frac {g_\varepsilon}{g_\varepsilon(x)}\right) =   \frac{1}{\delta} \big[ \nu(1+\delta)-\nu(1) \big] \cdot  \frac {f-f(x)}{g_\varepsilon}
\stackrel{\varepsilon \to 0}{\longrightarrow} \nu'(1)(f-f(x)) = \psi''(1)(f-f(x))
\]
and hence,
\begin{eqnarray*}
\frac 1 {\varepsilon^2} \Delta\left[\nu \left( \frac {g_\varepsilon} {g_\varepsilon(x)} \right) \Delta g_\varepsilon \right](x)
&\stackrel{\varepsilon \to 0}{\longrightarrow}&  \Delta \left[ \psi''(1)(f-f(x)) \Delta f \right] (x) \\
&=& -\psi''(1)\left[ f \Delta \Delta f - \Delta (f \Delta f)\right](x).
\end{eqnarray*}
Putting together the three limits calculated above yields
\begin{eqnarray*}
\frac {2g_\varepsilon} {\varepsilon^2}  \Gamma_2^\psi(1+ \varepsilon f) 
&\stackrel{\varepsilon \to 0}{\longrightarrow}   &      -\psi''(1) \left[ \Delta  \Gamma (f) + (\Delta f)^2  +   f \Delta \Delta f - \Delta (f \Delta f)        \right] \\
&=& -\psi''(1) \left[ \Delta  \Gamma (f)  - 2 \Gamma(f,\Delta f)   \right]\\
&=& -\psi''(1) \cdot 2 \Gamma_2 (f).
\end{eqnarray*}
Since $g_\varepsilon \stackrel{\varepsilon \to 0}{\longrightarrow} 1$, we conclude 
$\frac {1} {\varepsilon^2}  \Gamma_2^\psi(1+ \varepsilon f) \stackrel{\varepsilon \to 0}{\longrightarrow}  -\psi''(1) \Gamma_2 (f) $ which finishes the proof.
\end{proof}

\begin{corollary}
Let $\psi \in C^2(\posR)$ be concave with $\psi''(1)\neq 0 \neq \psi'(1)$ and let $d \in \posR$. Let $G=(V,E)$ be a graph satisfying the $CD\psi(d,0)$ condition.
Then, $G$ also satisfies the $CD\left(\frac{ -\psi''(1)}{\psi'(1)^2}d,0 \right)$ condition.
\end{corollary}
\begin{proof}
Let $f \in C(V)$. Since $G$ satisfies the $CD\psi(d,0)$ condition, we have
\[
-\psi''(1)\Gamma_2(f)=\lim_{\varepsilon \to 0} \frac 1 {\varepsilon^2} \Gamma_2^\psi(1+\varepsilon f) \geq  \lim_{\varepsilon \to 0}  \frac 1 {d \varepsilon^2}\left[ \Delta^\psi(1+\varepsilon f)\right]^2 = \frac {\psi'(1)^2 }{d} (\Delta f)^2.
\]
Since $\psi$ is concave and $\psi''(1) \neq 0$, one has $-\psi''(1)>0$. Thus, we obtain that $G$ satisfies the $CD\left(\frac{ -\psi''(1)}{\psi'(1)^2}d,0 \right)$ condition.
\end{proof}

\begin{example}
On finite graphs, the $CD\log(d,0)$ condition implies the $CD(d,0)$ condition.
\end{example}

The concept of the $\psi$-Laplacian and the $CD\psi$-condition can be extended to infinite graphs. In \cite{Hua2014}, Hua and Lin show that the $CD(d,0)$ condition, together with a geometric completeness property, implies stochastic completeness. Thus, we obtain that the $CD\psi$ condition on infinite graphs also implies stochastic completeness.


\section{Li-Yau inequalities} \label{CLiYau}

In this section, we will give two proofs of the $\psi$-Li-Yau inequality.
First, we will prove the $\psi$-Li-Yau inequality via a maximum principle which was used in \cite{Li1986, Bauer2013}. After that, we will give another characterization of the $CD \psi$ condition by using semigroup methods introduced in \cite{Bakry2006}. This characterization turns out to be a stronger version of the $\psi$-Li-Yau inequality.

  \subsection{A proof via the maximum principle}  

First, we present a monotonicity lemma which can be found in \cite[Lemma 4.1]{Bauer2013}.
This gives an important maximum property of the heat operator.
\begin{lemma} [Monotonicity of $\LL$] 
Let $G = (V,E)$ be a graph and let $g,F : V \times (0,T] \to \IR$ be differentiable functions, such that $g\geq 0$ and such that $F$ attains a local maximum in some $(x_0,t_0)$.
Then, one has
\begin{equation}
  \LL(g)F (x_0,t_0) \geq \LL (gF)(x_0,t_0). \label{MON}
\end{equation}
\end{lemma}

\begin{proof}
A short calculation gives
\begin{eqnarray*}
       \Delta(g)F (x_0,t_0)
&=&    \sum_{y \sim x_0} g(y,t_0)F(x_0,t_0) - g(x_0,t_0)F(x_0,t_0)\\
&\geq& \sum_{y \sim x_0} g(y,t_0)F(y,t_0) - g(x_0,t_0)F(x_0,t_0) \\
&=&    \Delta (gF)(x_0,t_0).
\end{eqnarray*}

Similarly,
\[
F \partial_t g (x_0,t_0) \leq F \partial_t g (x_0,t_0) + g \partial_t F (x_0,t_0) =  \partial_t (gF) (x_0,t_0)
\] 
since $\partial_t F =0$ if $t_0 \in (0,T)$ and $\partial_t F \geq0$ if $t_0=T$.
The difference of these estimates yields the claim.
\end{proof}

\begin{theorem}[$\psi$-Li-Yau inequality] \label{TPLY}
Let $G = (V,E)$ be a finite graph satisfying $CD\psi(D,0)$ and $\sol$ a positive solution to the heat equation on $G$ . Then for all $x\in V$ and $t \in \posR$, one has 
\[
-\Delta^\psi u(x,t) \leq \frac {d} {2 t}.
\]
\end{theorem}

\begin{proof}

We define for $x \in V$ and $t \in \nnegR$ 
\[F(x,t) := -t \Delta^\psi u(x,t).\]
It is sufficient to show that for all $T>0$, one has
\[
\sup_{x\in V, 0\leq t\leq T} F(x,t) \leq \frac {d} {2}.
\]
Since $V \times [0,T]$ is compact and $F$ is continuous, the restriction $F\big|_{V\times [0,T]}$ attains its maximum in some $(x_0,t_0)$.
We assume without loss of generality that $F(x_0,t_0)$ is positive. Since $F(\cdot,0)=0$, we can deduce $t_0 > 0$ and thus, the maximum is attained on $V \times (0,T]$. 
Hence, we can use the estimate (\ref{CCD}) from the characterization of $CD\psi$ and the estimate (\ref{MON}) from the monotonicity lemma with $g(\cdot,t) := \frac u t$ for all $t \in(0,T]$.
The following computation is understood to take place at the point $(x_0,t_0)$. We obtain

\begin{eqnarray*}
 		  									  											\frac u {t_0^2} F 
&=&																							\LL \left(g\right) F
\stackrel{(\ref{MON})}{\geq}										\LL \left( gF \right)  
= 																							\LL \left(-u\Delta^\psi u  \right)  
\stackrel{(\ref{CCD})}{\geq} 										\frac {2}{d} u (\Delta^\psi u)^2
=  																							\frac {2}{d} \frac u {t_0^2} F^2. 
\end{eqnarray*}

Since $F(x_0,t_0) >0$, we can conclude $F \leq \frac {d} {2}$.
\end{proof}

  \subsection{A proof via semigroup methods}

Next, we use semigroup methods to give another proof of the $\psi$-Li-Yau inequality.
Moreover, we give another characterization of the $CD\psi$ condition which is inspired by a similar result that Bakry and Ledoux showed on diffusion  semigroups (cf. \cite{Bakry2006}). 

\begin{defn}[Operator semigroup] 
Let $G=(V,E)$ be a finite graph. The Laplacian $\Delta$ generates an \emph{operator semigroup} $\left(P_t\right)_{t\geq 0}: C(V) \to C(V)$ with 
\[
P_t f := e^{\Delta t} f := \sum_{k=0}^{\infty} \frac {t^k \Delta^k f}{k!} \mbox{ for all } t \in \nnegR \mbox{ and all } f \in C(V). 
\]
\end{defn} 
 
\begin{proposition}[Basic properties of operator semigroups] {\label{SG}}
Let $G=(V,E)$ be a finite graph with Laplacian $\Delta$ and the generated semigroup $\left(P_t\right)_{t\geq 0}$.
\eBr
  \item
		The semigroup satisfies the property 
		
		${P_t ( P_s f ) = P_{t+s} f}$ for all $f \in C(V)$ and all $s,t \geq 0$. 
  \item
	  The semigroup gives a solution to the heat equation. 
	  I.e. $P_t f$ is continuous in $t$ for all $t \geq 0$ and $f \in C(V)$, and furthermore  $\LL ( P_t f) = 0$ for all $t \geq 0$ and all $f 			\in C(V)$.
	\item  
	  The semigroup satisfies the initial condition $P_0 f = f$ for all $f \in C(V)$.
  \item
		If the initial condition $f \in C(V)$ is positive, then $P_t f \in C(V)$ is also positive for all $t \geq 0$.
\end{enumerate}

\end{proposition}

  This proposition is a standard one and a proof can be found e.g. in \cite{Chung2000}.

\begin{theorem}[Semigroup form of the $\psi$-Li-Yau inequality] 
Let $\psi \in C^1(\posR)$, and let $G=(V,E)$ be a finite graph with Laplacian $\Delta$ which generates the semigroup $\left(P_t\right)_{t\geq 0}$. 
Then, the following statements are equivalent.
\eChar
	\item $G$ satisfies $CD\psi({d},0)$.
	\item For all positive functions $f \in C^+(V)$ and all $t\geq 0$, one has
	  \[
	    P_t f \Delta^\psi P_t f \geq P_t(f \Delta^\psi f)\left(1 + \frac {{2} t} n \Delta^\psi P_t f  \right).
	  \]
\end{enumerate}

If one of these statements is true, then the $\psi$-Li-Yau inequality
\[
-\Delta^\psi P_t f \leq \frac n {2 t}
\]
holds for all $f \in C^+(V)$ and all $t>0$.

\end{theorem}

The proof of this theorem is in the spirit of Bakry and Ledoux (cf. \cite[Theorem 1]{Bakry2006}). In contrast to their proof, we are able to bypass the chain rule by using Theorem \ref{tchar} which is the characterization of the $CD\psi$ condition. It seems to be quite a challenge to obtain logarithmic Sobolev inequalities (cf. \cite[inequalities (1.10), (1.11)]{Bakry2006}) on graphs by using our methods. Nevertheless in \cite{Chung2000}, Chung, Grigor'yan and Yau have shown another version of the Sobolev inequality on manifolds and on graphs.

\begin{proof}

First, we show $\emph{(\ref{2})} \Rightarrow \emph{(\ref{1})}$. 
We will use identity (\ref{repGG}) from the representation lemma of $\Gamma_2^\psi$.
For all $f \in C^+(V)$, the assumption $\emph{(\ref{2})}$ implies 
\begin{eqnarray*}
0
&\leq& \lim_{t \to 0} \frac 1 {2t} \left[ P_tf \Delta^\psi P_t f - P_t (f\Delta^\psi f) - \frac {{2} t} n P_t(f \Delta^\psi f) \Delta^\psi P_t f \right] \\
&=&    \lim_{t \to 0} \frac 1 {2t} \left[ (P_tf \Delta^\psi P_t f - P_0f \Delta^\psi P_0 f) - (P_t (f\Delta^\psi f - P_0 (f\Delta^\psi f))\right] \\&& - \frac {{1}} {n} f (\Delta^\psi f)^2\\
&=&    \frac 1 2 \partial_t\left[ P_tf \Delta^\psi P_t f - P_t (f\Delta^\psi f) \right]_{t=0}   - \frac {{1}} {n} f (\Delta^\psi f)^2 \\
&=&    \frac 1 2 \left[ \partial_t( P_tf \Delta^\psi P_t f) -  \Delta P_0 (f\Delta^\psi f) \right]_{t=0}   - \frac {{1}} {n} f (\Delta^\psi f)^2 \\
&=&    \frac 1 2 \left[ \partial_t( P_tf \Delta^\psi P_t f) -  \Delta  (P_t f\Delta^\psi P_t f) \right]_{t=0}   - \frac {{1}} {n} f (\Delta^\psi f)^2 \\
&=&    \frac 1 2 [\LL (-P_tf \Delta^\psi P_t f)]_{t=0} - \frac {{1}} {n} f (\Delta^\psi f)^2 \\
&\stackrel{(\ref{repGG})}{=}&
								 \left[P_tf {\Gamma_2^\psi} (P_tf)  \right]_{t=0} - \frac {{1}} {n} f (\Delta^\psi f)^2 \\
&=&    f  {\Gamma_2^\psi} (f) - \frac {{1}} {n} f (\Delta^\psi f)^2.
\end{eqnarray*}
Thus,
\[
{\Gamma_2^\psi} (f) \geq \frac {{1}} {n} (\Delta^\psi f)^2.
\]
Since $f$ is arbitrary, we obtain the $CD\psi(n,0)$ inequality.

Next, we show $\emph{(\ref{1})} \Rightarrow \emph{(\ref{2})}$.

Let $f \in C(V)$ and $t>0$.
Following the notation of Bakry and Ledoux, for $s \in [0,t]$, we denote
\begin{eqnarray*}
	g	    &:=&	P_s f,\\
	A	    &:=& g\Delta^\psi g,\\
\phi(s) &:=& P_{t-s}(P_s f \Delta^\psi P_s f) = P_{t-s}(A).
\end{eqnarray*}

We take the derivative of $\phi$. By using identity (\ref{repGG}) from the representation lemma of $\Gamma_2^\psi$, we obtain

\begin{eqnarray*}
\phi'(s)
&=& \partial_s  P_{t-s}(A)
= -\Delta  P_{t-s}(A) + P_{t-s}(\partial_s A)
= P_{t-s}\LL (-A) = P_{t-s}\LL (-g\Delta^\psi g)\\
&\stackrel{(\ref{repGG})}{=}& P_{t-s}(2 g {\Gamma_2^\psi}  (g))
\stackrel{CD\psi}{\geq}\frac {{2}} n   P_{t-s}\left(   g \left(\Delta^\psi g\right)^2 \right) =  \frac {{2}} n   P_{t-s}\left(\frac {A^2} g \right)\\
&\geq&   \frac {{2}} n   \frac {(P_{t-s}A)^2}{P_{t-s} g}
= \frac {{2}} n  \frac {\phi(s)^2}{P_t f} \geq 0.
\end{eqnarray*}
In the first inequality, we also used the positivity of $P_{t-s}$.

This calculation implies
\begin{equation}
P_tf \Delta^\psi P_t f = \phi(t)\geq \phi(0) = P_t(f \Delta^\psi f).\label{ptpo}
\end{equation}

Suppose $\Delta^\psi P_t f \geq 0$ and $P_t(f \Delta^\psi P_t f)\leq 0$. Then,
the claim is obvious.

Suppose not. Since $\phi$ is monotonically non-decreasing, we can conclude $\phi(s) \neq 0$ for all $s \in [0,t]$ . Thus, we obtain
\[
     - \left( \frac 1 {\phi} \right)'
=    \frac {\phi'}{\phi^2} 
\geq \frac {{2}} {n P_t f}.
\]
By integrating this identity from $0$ to $t$, we get
\[
\frac 1 {\phi(0)} -  \frac 1 {\phi(t)} \geq  \frac {{{2}} t} {n P_t f}.
\]

Since $\phi(0)$ and $\phi(t)$ have the same sign, we see
\[
{\phi(t)} -  {\phi(0)} \geq  \phi(t) \phi(0)  \frac {{{2}}t} {n P_t f}.
\]
Substituting $\phi(t)$ and $\phi(0)$ into this estimation yields the result.

Now, we will deduce the $\psi$-Li-Yau inequality from the equivalent statements $\emph{(\ref{1})}$ and $\emph{(\ref{2})}$.
Suppose $\Delta^\psi P_t f > 0$. Then, the claim is obvious. 

Suppose not. We have already seen in inequality (\ref{ptpo}) that $\emph{(\ref{1})}$ implies $P_tf \Delta^\psi P_t f \geq P_t(f \Delta^\psi f)$. Thus,
\[
0 \geq P_tf \Delta^\psi P_t f \geq P_t(f \Delta^\psi f).
\]
Consequently, $\emph{(\ref{2})}$ implies
\[
1 + \frac {{2} t} n \Delta^\psi P_t f  \geq 0
\]
which is equivalent to the Li-Yau inequality
\[
-\Delta^\psi P_t f \leq \frac n {{2} t}.
\]
This finishes the proof.
\end{proof}

\section{Harnack inequalities} \label{CHarnack}  
 
The Harnack inequality states that if $u$ is a solution to the heat equation and if $(x_i,t_i)_{i=1,2}$ are two points in space-time, then $\frac{u(x_1,t_1)}{u(x_2,t_2)}$ can be estimated by a function only depending on a certain distance of the points $(x_1,t_1)$ and $(x_2,t_2)$.

In this section, we show in which sense the $\psi$-Li-Yau inequality can be understood as a gradient estimate and how to use this property to derive Harnack type inequalities.
To do so, we will use the methods introduced in \cite{Bauer2013} which turn out to be applicable here with minor changes.
We will note in this section that the concavity of $\psi$ is crucial to derive Harnack inequalities from the $\psi$-Li-Yau inequality.
Furthermore, we will see in Subsection \ref{SRF} that we need concavity of $\psi$ to prove  $CD\psi$ inequalities on Ricci-flat graphs.

   \subsection{Preliminaries}

To understand the $\psi$-Li-Yau inequality as a gradient estimate, we use the $\psi$-gradient $\Gamma^\psi$ (see Subsection \ref{SGF}).
This is defined as 
$\Gamma^\psi = \Delta^{\overline{\psi}}$ with
$$\overline{\psi}(x)= \psi'(1)\cdot(x-1)  - (\psi(x) - \psi(1))$$ for all $x>0$ and all concave $\psi \in C(\posR)$.
We recall the gradient representation of $\Delta^\psi$ (cf. Lemma \ref{LGR}).
Let $G=(V,E)$ be a finite graph, and let $\psi \in C^1(\IR)$ be a concave function. Then, the gradient representation of $\Delta^\psi$ states that for all $f \in C^+(V)$, one has
\[
-\Delta^\psi f = \Gamma^\psi (f) - \psi'(1) \frac {\Delta f} f.
\]

This formulation will turn out to be a convenient basis for deducing Harnack inequalities from the Li-Yau inequality.
We only need to introduce one more constant dependent on $\psi$.

\begin{defn}[Harnack constant] \label{DHC} 
Let $\psi \in C^1(\posR)$ be a concave function. We define the \emph{Harnack-constant} $H_\psi$ of $\psi$ as
\begin{equation}
H_\psi := \sup_{x>1} \frac{(\log x)^2}{\overline{\psi}(x)} \in [0,\infty].
\end{equation}
\end{defn}

This constant is defined to give the connection between  $\psi$-Li-Yau type gradient estimates and Harnack inequalities.
The case $H_\psi = \infty$ is allowed, but this only occurs if $\psi''(1)=0$ (cf. Lemma \ref{LDH}).

The following lemma is the key to prove Harnack inequalities by using the methods introduced in \cite[Section 5]{Bauer2013} and gives the link between the $\Gamma^\psi$-calculus and these methods.

\begin{lemma} [Estimate of ${\Gamma^\psi}$] 
Let $\psi \in C^1(\posR)$ be a concave function,
let $G=(V,E)$ be a graph, let $f \in C^+(V)$ and let $v,w \in V$ with $v \sim w$.
Then, 
\begin{equation}
\log {\frac {f(w)}{f(v)}} \leq \sqrt{H_\psi} \sqrt{\Gamma^\psi (f)}(v) \label {EG}.
\end{equation}

\end{lemma}

\begin{proof}
First, we show the following claim. 	For all $x \in \posR$, one has
\begin{equation}
\log x \leq \sqrt{H_\psi}\sqrt{\overline{\psi}(x)}. \label{claimG}
\end{equation}

This claim is obvious for $x \leq 1$, since we obtain $\log x \leq 0$ in this case.
If $x > 1$, the claim follows from the definition of $H_\psi$.

For the assertion of the lemma, we use $\overline{\psi} \geq 0$ and apply inequality (\ref{claimG}) proven above with $x = \frac {f(w)}{f(v)}$ to compute 
\begin{eqnarray*}
			  \sqrt{H_\psi} \sqrt{\Gamma^\psi (f)}(v)
&=&			\sqrt{H_\psi} \sqrt{\Delta^{\overline{\psi}} f}(v) \\  
&\stackrel{(\ref{repG})}{=}&     \sqrt{H_\psi} \sqrt{\sum_{\widetilde v \sim v} \overline{\psi}\left(\frac{f(\widetilde v)}{f(v)}   \right)}\\
&{\geq}&  \sqrt{H_\psi} \sqrt{ \overline{\psi}\left(\frac{f(w)}{f(v)} \right)  }\\
&\geq&   \log {\frac {f(w)}{f(v)}}.
\end{eqnarray*}
This finishes the proof.
\end{proof}

Next, we give a lemma which is a special case of \cite[Lemma 5.3]{Bauer2013}.

\begin{lemma}[Minimal integral estimate] \label{MIL} 
Let $T_1,T_2 \in \IR$ with $T_2 > T_1$.
Let $\gamma: [T_1,T_2] \to \nnegR$ be a continuous function and let $C_1, C_2 \in \posR$ be positive constants. Then, one has 
\begin{equation}
 \frac {C_2^2}{C_1(T_2-T_1)} \geq \inf_{s\in [T_1,T_2]} \left( C_2 \sqrt{\gamma(s)} - C_1 \int_s^{T_2} \gamma(t) dt \right) . \label{HL}
\end{equation}

\end{lemma}
We include the proof for convenience of the reader.
\begin{proof}
We estimate the infimum by an average integral with the weight function $\phi:[T_1,T_2] \to \posR, s \mapsto s-T_1$.

\begin{eqnarray*}
		&&		\inf_{s\in [T_1,T_2]} \left( -C_1 \int_s^{T_2} \gamma(t) dt + C_2 \sqrt{\gamma(s)} \right) \\
&\leq&  \frac{ \int_{T_1}^{T_2} \phi(s) \left(-C_1 \int_s^{T_2} \gamma(t) dt + C_2 \sqrt{\gamma(s)}  \right) ds }{\int_{T_1}^{T_2} \phi(s)ds}  \\
&=&     \frac 1 {(T_2-T_1)^2}   \int_{T_1}^{T_2} \left( - 2 C_1  \gamma(s) \int_{T_1}^s \phi(t)       dt +   2 C_2(s-T_1)\sqrt{\gamma(s)}     \right) ds\\
&=&     \frac 1 {(T_2-T_1)^2}   \int_{T_1}^{T_2} \left( -  C_1  \gamma(s) (s-T_1)^2 +   2 C_2(s-T_1)\sqrt{\gamma(s)}     \right) ds \\
&\leq&  \frac 1 {(T_2-T_1)^2}   \int_{T_1}^{T_2} \frac{C_2^2}{ C_1} ds  \\
&=&     \frac {C_2^2}{C_1(T_2-T_1)}.
\end{eqnarray*}

In the last estimate, we used $-C_1 x^2 + 2C_2 x \leq \frac{C_2^2}{ C_1}$ with $x = (s-T_1)\sqrt{\gamma(s)}$.
\end{proof}

   \subsection{Main theorem}

The following theorem is in the spirit of Bauer, Horn, Lin, Lippner, Mangoubi, and Yau (cf. \cite[Theorem 5.1]{Bauer2013}).
In contrast to their version and in order to focus on our new methods, we do not consider the effect of potentials but instead, we consider more general gradient forms.

To prove the Harnack inequality, we require the gradient estimate 
\[
D_1 \Gamma^\psi (u) (x,t) - \partial_t \log u (x,t) \leq \frac{D_2}t + D_3 
\]
for all positive solutions $\sol$ to the heat equation and some positive constants $D_1,D_2,D_3 \in \posR$.
By using the gradient representation (Lemma \ref{LGR})
\[
-\Delta^\psi f = \Gamma^\psi (f) - \psi'(1) \frac {\Delta f} f,
\]
we will be able to guarantee the required gradient estimate due to the $\psi$-Li-Yau inequality (Theorem \ref{TPLY})
\[
-\Delta^\psi u(x,t) \leq \frac {d} {2 t}
\]
for graphs satisfying the $CD\psi(d,0)$ condition.

\begin{theorem}[Harnack inequality as a consequence of a $\psi$-gradient estimate] \label{THarnack} 
Let G=(V,E) be a finite and connected graph, $D_1,D_2,D_3 \in \posR$ positive constants, and let $\sol$ be a function satisfying
\begin{equation}
D_1 \Gamma^\psi (u) (x,t) - \partial_t \log u (x,t) \leq \frac{D_2}t + D_3  \label{Hreq} 
\end{equation}
for all $x\in V$ and $t \in \nnegR$. 
Then, 
\[
\frac{u(x_1,T_1)}{u(x_2,T_2)} \leq  \left( \frac{T_2}{T_1} \right)^{D_2} \exp(D_3(T_2-T_1)) \exp\left(\frac {H_\psi d(x_1,x_2)}{ D_1 (T_2-T_1)}\right)
\]
holds for all $x_1,x_2 \in V$ and all positive $T_1<T_2$.
\end{theorem}

The proof of this theorem is very similar to the one given in \cite[Theorem 5.1]{Bauer2013}.

\begin{proof}
First, we consider the case $x_1 \sim x_2$. Let $T_1 < T_2$ and $s \in \left[ T_1,T_2 \right]$.
We use the assumption (\ref{Hreq}) of the theorem and the estimate (\ref{EG}) of $\Gamma^\psi$
to estimate
\begin{eqnarray*}
		  \log \frac {u(x_1,T_1)} {u(x_2,T_2)} 
&=&        \log \frac {u(x_1,T_1)} {u(x_1,s)} + \log \frac {u(x_1,s)} {u(x_2,s)} + \log \frac {u(x_2,s)} {u(x_2,T_2)} \\
&=&        \int_{T_1}^s -\partial_t \log u(x_1,t) dt + \log \frac {u(x_1,s)} {u(x_2,s)}  -\int_{s}^{T_2} \partial_t \log u(x_2,t) dt \\
&\stackrel{(\ref{Hreq})}{\leq}& 
				   \int_{T_1}^s \left( \frac {D_2}t + D_3 - D_1 \Gamma^\psi (u)(x_1,t) \right) dt + \log \frac {u(x_1,s)} {u(x_2,s)}  
				\\&& + \int_{s}^{T_2} \left( \frac {D_2}t + D_3 - D_1 \Gamma^\psi (u)(x_2,t) \right) dt \\
&\stackrel{D_1 \Gamma^\psi \geq 0}{\leq}& 
			 	\int_{T_1}^{T_2} \left( \frac {D_2}t + D_3 \right) dt +   \log \frac {u(x_1,s)} {u(x_2,s)} - \int_{s}^{T_2} D_1  \Gamma^\psi  (u)(x_2,t) dt \\
&\stackrel{(\ref{EG})}{\leq}&	
				D_2\log\ \frac {T_2}{T_1} + D_3(T_2-T_1) \\&& +    \sqrt{H_\psi}\sqrt{\Gamma^\psi (u)(x_2,s)}   - \int_{s}^{T_2} D_1  \Gamma^\psi (u)(x_2,t) dt. \\		 	
\end{eqnarray*}

Now, we take the infimum over all $s \in \left[ T_1,T_2 \right]$ and, by using the minimal integral estimate (\ref{HL}), we obtain

\begin{eqnarray*}
		  \log \frac {u(x_1,T_1)} {u(x_2,T_2)} 
&\leq& 
				D_2\log\ \frac {T_2}{T_1} + D_3(T_2-T_1) \\&& +  \inf_{s\in[T_1,T_2]} \left(  \sqrt{H_\psi}\sqrt{\Gamma^\psi (u)(x_2,s)}   - \int_{s}^{T_2} D_1  \Gamma^\psi (u)(x_2,t) dt \right) \\		 	
&\stackrel{(\ref{HL})}{\leq}&	
				D_2\log\ \frac {T_2}{T_1} + D_3(T_2-T_1) + \frac {H_\psi} {D_1(T_2-T_1)}. \\		 	
\end{eqnarray*}
The minimal integral estimate (\ref{HL}) was applied with  $\gamma(t) = \log u(x_2,t)$, and $C_2 = \sqrt{H_\psi}$, and $C_1 = D_1$.

Now, we consider the general case (i.e. the vertices $x_1$ and $x_2$ are not necessarily adjacent). Denote $\dd:= d(x_1,x_2)$. Since $G$ is connected, there is a path $x_1=v_0 \sim \ldots \sim v_{\dd} = x_2$ of length $\dd$, and there are positive numbers $T_1 = t_0 < \ldots < t_{\dd} = T_2$ with $t_i - t_{i-1} = \frac {T_2 - T_1}{\dd}$ for $i\in \{1,\ldots,\dd\}$.
Using the estimation from above yields
\begin{eqnarray*}
		  \log \frac {u(x_1,T_1)} {u(x_2,T_2)} 
&=& \sum_{i=1}^{\dd}  \log \frac {u(v_{i-1},t_{i-1})} {u(v_i,t_i)} \\
&\leq& \sum_{i=1}^{\dd}  D_2\log\ \frac {t_i}{t_{i-1}} + D_3(t_i-t_{i-1}) + \frac {H_\psi} {D_1(t_i-t_{i-1})} \\
&=&  D_2 \log \frac {T_2}{T_1} +  D_3(T_2-T_1) + \frac {H_\psi d(x_1,x_2)^2} {D_1(T_2-T_1)}. \\		
\end{eqnarray*}
Hence,
\[
\frac{u(x_1,T_1)}{u(x_2,T_2)} \leq  \left( \frac{T_2}{T_1} \right)^{D_2} \exp(D_3(T_2-T_1)) \exp\left(\frac {H_\psi d(x_1,x_2)}{ D_1 (T_2-T_1)}\right)
\]
which is the claim of the theorem.
\end{proof}

\begin{corollary}[Harnack inequalities as a consequence of the $CD\psi$ condition] 
Let $\psi \in C^1(\posR)$ be a concave function with $\psi'(1)=1$, and let $G=(V,E)$ be a graph satisfying the $CD\psi(d,0)$ inequality for some $d \in \posR$. Then for all positive solutions $\sol$ to the heat equation on $G$, all $x_1,x_2 \in V$, and all positive $T_1<T_2$, one has
\[
\log \frac{u(x_1,T_1)}{u(x_2,T_2)} \leq \frac d {2} \log \frac{T_2}{T_1} + \frac{H_\psi d(x_1,x_2)^2}{(T_2-T_1)}.
\]

\end{corollary}

\begin{proof}
Since $G$ satisfies the $CD\psi(d,0)$ inequality, the $\psi$-Li-Yau inequality holds. Thus,
\[
\Gamma^\psi (u) - \frac {\Delta u} u = -\Delta^\psi u \leq \frac d {2t}.
\]
Hence, we can apply the Harnack inequality with $D_1=1$, $D_2 = \frac d {2}$, and $D_3 = 0$, and we obtain the claim.
\end{proof}

\section{Examples} \label{CExamples}

To show that the Harnack inequality presented here is stronger than the one given in \cite[Corollary 5.2]{Bauer2013}, we have to assure a unified context for both. Unfortunately, the different curvature-dimension conditions seem to be not unifiable. Instead, so called Ricci-flat graphs will turn out to be an appropriate basis to compare the Harnack inequalities.
To do this comparison, we need some preliminary considerations.
In the following subsection, we will show that Ricci-flat graphs satisfy the $CD\psi(d,0)$ condition for suitable $\psi$. For a given Ricci-flat graph, the dimension bound $d$ depends on the parameter $\psi$. This dependence can be described by a constant $C_\psi$.
In the second subsection of this section, we compute this constant $C_\psi$ and the Harnack constant $H_\psi$ (cf. Definition \ref{DHC}) for several $\psi$.
We use these computations to discuss the announced break of analogy between the $CDE$ condition introduced in \cite[Definition 3.9]{Bauer2013} and the $CD$ condition (cf. Definition \ref{DCD}).
In the third subsection, we establish a Harnack inequality on Ricci-flat graphs. Finally, we will compare this inequality with the one given in \cite[Corollary 5.2]{Bauer2013}. 

  \subsection{Ricci-flat graphs}\label{SRF} 

Ricci-flat graphs were introduced by Chung and Yau \cite{Chung1996} as a generalization of Abelian Cayley graphs to prove Harnack inequalities and log-Sobolev inequalities.
These graphs have been the basis to establish new notions of Ricci curvature on graphs (cf. \cite{Bauer2013, Lin2010}).

The goal of this subsection is to prove that Ricci-flat graphs satisfy the $CD \psi$ inequality with curvature bound zero. This subsection is in the spirit of \cite[Subsection 6.3]{Bauer2013}, where the $CDE$ condition is proved for Ricci-flat graphs.

\begin{defn}[Ricci-flat graphs] 
Let $D \in \IN$. A finite graph $G = (V,E)$ is called $D$\emph{-Ricci-flat} in $v \in V$ if all $w \in N(v):=\{v\} \cup \{w \in V: w \sim v\}$ have the degree $D$, and if there are maps $\eta_1,\ldots,\eta_D : N(v) \to V $, such that for all $w \in N(v)$ and all $i, j \in \{1,\ldots,D\}$ with $i \neq j$, one has
\begin{eqnarray}
\eta_i(w) &\sim& w,  \label{r1} \\
\eta_i(w) &\neq& \eta_j(w), \label{r2}\\
\bigcup_k \eta_k(\eta_i(v)) &=& \bigcup_k \eta_i(\eta_k(v)) \label{r3}.
\end{eqnarray}
The graph $G$ is called $D$\emph{-Ricci-flat} if it is $D$-Ricci-flat in all $v \in V$.

\end{defn}

\begin{example}
All Abelian Cayley graphs with degree $D$ are $D$-Ricci-flat as mentioned already in \cite{Chung1996}.
\end{example}

In the next lemma, we collect some facts which are already used in \cite{Bauer2013}.

\begin{lemma} [Basic properties of Ricci-flat graphs] \label{Lj} 
Let $G = (V,E)$ be Ricci-flat in $v \in V$ with according maps $\eta_1,\ldots,\eta_D : N(v) \to V $.
  \eBr
    \item
      Let $f \in C(V)$ be a function. Then for all $i \in \{1,\ldots,D\}$, one has
      \begin{equation}
        \sum_k f(\eta_k \eta_i(v)) = \sum_k f(\eta_i \eta_k(v)). \label{komm}
      \end{equation}
  
	  \item 
       For all $i \in \{1,\ldots,D\}$, there is a unique $j = \ii$, such that 
       $\eta_{i} (\eta_j(v)) = v$.
       Additionally, the map $i \mapsto \ii$ is a permutation of $\{1,\ldots,D\}$.

  \end{enumerate}
\end{lemma}

\begin{proof}

First, we prove \emph{(1)}.

Since $\bigcup_k \eta_k(\eta_i(v)) = \bigcup_k \eta_i(\eta_k(v))$, it is sufficient to show that no vertex in (\ref{komm}) is summed up twice.
This is clear if $\bigcup_k \eta_k(\eta_i(v))$ and if $\bigcup_k \eta_i(\eta_k(v))$ are disjoint unions.
We know, $\bigcup_k \eta_k(\eta_i(v))$ is a disjoint union by (\ref{r2}). Since $D < \infty$, identity (\ref{r3}) implies that $\bigcup_k \eta_i(\eta_k(v))$ is a disjoint union.

Next, we prove \emph{(2)}.

The uniqueness of $\ii$ is obvious. 

Suppose $i \mapsto \ii$ is not a permutation.
Then, there are $i,k \in \{1,\ldots,D\}$ with $i \neq k$ and $j = \ii = j(k)$. This means
$\eta_i\eta_j(v) = \eta_k\eta_j(v)$. Hence,
\[
\# \bigcup_l \eta_l \eta_j (v) < D.
\]
This is a contradiction.
\end{proof}

To prove a $CD\psi$ inequality on Ricci-flat graphs, we need to introduce the constant $C_\psi$.

\begin{defn}\label{DCP} 
Let $\psi \in C^1(\posR)$. Then for all $x,y > 0$, we write
\[
\widetilde{\psi}(x,y):= \left[\psi'(x) + \psi'(y) \right](1-xy) + x [\psi(y) - \psi(1/x)] + y [\psi(x) - \psi(1/y)]
\]
and
\[
C_\psi := \inf_{x,y>0} \frac{\widetilde{\psi}(x,y)}{(\psi(x) + \psi(y) - 2\psi(1))^2} \in [-\infty, \infty].
\]

\end{defn}

To obtain useful results, we need $C_\psi>0$. We will discuss this condition in the next subsection.

\begin{rem}
For $\psi=\log$, it suffices to consider the infimum of the function in $C_\psi$ for the diagonal, i.e., $x=y$. This behavior is also indicated by numerical computations for various concave $\psi$ and it would be interesting to know whether this behavior can be established for arbitrary concave $\psi$.
\end{rem}

\begin{theorem}[$CD\psi$ for Ricci-flat graphs] \label{TRicci} 
Let $D \in \IN$, let $G=(V,E)$ be a $D$-Ricci-flat graph, and let $\psi \in C^1(\posR)$ be a concave function, such that $C_\psi>0$. Then, $G$ satisfies the $CD \psi (d,0)$ inequality with $d = D / C_\psi$.
\end{theorem}

The proof and the notation are inspired by the proof of Theorem 6.7 in \cite{Bauer2013}.

\begin{proof}

We have to show for all $f \in C(V)$ and all $v \in V$ that
\[
2 \Gamma_2^\psi (f) (v) \geq \frac {2 C_\psi}{D} \left[\Delta^\psi f (v)\right]^2.
\]	

First, we recall the definitions
\begin{eqnarray*}
\Delta^\psi f(v)          &=&  \sum_{w\sim v} \psi\left( \frac{f(w)}{f(v)} \right) - \psi (1),\\
\Omega^\psi f(v) &=&  \sum_{w\sim v} \psi'\left( \frac{f(w)}{f(v)} \right) \cdot \frac{f(w)}{f(v)} \cdot \left[ \frac{\Delta f(w)}{f(w)} -   \frac{\Delta f(v)}{f(v)} \right],  \\
2 \Gamma_2^\psi (f) 			&=& \Omega^\psi f + \frac {\Delta f \Delta^\psi f} f - \frac {\Delta \left(f \Delta^\psi f\right)} f.
\end{eqnarray*}	

We can assume $$\psi(1)=0$$ without loss of generality since $\Gamma_2^\psi$, $\Delta^\psi$ and $C_\psi$ are invariant under adding constants to $\psi$.
Let $v \in V$ and $f \in C(V)$. Since $G$ is Ricci-flat, there are maps $\eta_1,\ldots,\eta_D : N(v) \to V$ as demanded in the definition. For all $i,j \in \{1,\ldots,D\}$, we denote
\begin{eqnarray*}
y 									&:=& f(v),\\
y_i  								&:=& f(\eta_i(v)),\\
y_{ij} 							&:=& f(\eta_j(\eta_i(v))),\\
z_i 						 		&:=& y_i / y, \\
z_{ij} 							&:=& y_{ij}/y_{i}.
\end{eqnarray*}

We use the representation of $\Delta^\psi$ (Lemma \ref{LRL}) to obtain the following two identities
\begin{eqnarray*}
\Delta^\psi f (v) 				&=& \sum_i \psi(z_i), \\
\Delta^\psi f (\eta_i(v)) &=& \sum_j \psi(z_{ij}). \\
\end{eqnarray*}


Thus, we can compute
\begin{eqnarray*}
&&\frac {\Delta \left(f \Delta^\psi f\right)} {f} (v) - \frac {\left(\Delta f \right) \Delta^\psi f} {f} (v)\\
&=&   \frac {\sum_{w \sim v}- f(v) \Delta^\psi f(v) + f(w) \Delta^\psi f(w) } {f (v)} - \frac {\left(\sum_{w\sim v} - f(v) + f(w)  \right) \sum_i \psi(z_i)} {f(v)}                     \\
&=& \left[ \sum_{i,j}  z_j \psi(z_{ji}) - \psi(z_i) \right] -  \left[ \sum_{i,j} \left( z_j - 1 \right) \psi (z_i)  \right]\\
&=& \sum_{i,j} z_j [\psi(z_{ji}) - \psi(z_i)]
\end{eqnarray*}
and
\begin{eqnarray*}
\Omega^\psi  f (v) &=& \sum_i \psi'(z_i)z_i \left[   \frac {(\Delta f)(\eta_i(v))}{y_i} -  \frac {(\Delta f)(v)}{y}           \right] \\
&=& 																\sum_{i,j} \psi'(z_i)z_i (z_{ij}-z_j) \\
&\stackrel{(\ref{komm})}{=}& 				\sum_{i,j} \psi'(z_i)z_j (z_{ji}-z_i).
\end{eqnarray*}
 
We summarize
\begin{equation}
  2 \Gamma_2^\psi (f) (v) = \sum_{i,j} z_j \left( \psi'(z_i)(z_{ji}-z_i) - [\psi(z_{ji}) - \psi(z_i)] \right) .\label{GRF}
\end{equation}

Since $\psi$ is concave, every summand is positive. As we showed in the second claim of Lemma \ref{Lj}, for each $i$, there is a unique $j=\ii$ with $\eta_i(\eta_j(v)) = v$.
Now, we disregard all other summands of (\ref{GRF}) and use $z_ {ji} = 1/z_{\ii}$ if $j=\ii$ to estimate

\begin{eqnarray*}
2 \Gamma_2^\psi (f) (v)&\geq& \sum_i z_{\ii}\left( \psi'(z_i)\left( \frac 1 {z_{\ii}} - z_i \right) -\left[ \psi\left( \frac 1 {z_{\ii}} \right) - \psi(z_i) \right]    \right) \nonumber  \\
&=& \sum_i \psi'(z_i) - \sum_i   z_{\ii} \psi\left( \frac 1 {z_{\ii}} \right) + \sum_i  z_{\ii}\left( \psi(z_i) - z_i \psi'(z_i) \right). 
\end{eqnarray*}

The next step is to symmetrize the sum. Unfortunately, we do not have $j(\ii)=i$ in general. But instead, we can use the rearrangement inequality. 
This states that for all permutations $\sigma$ on $\{1,\ldots,D\}$ and all $a_1 \leq \ldots \leq a_D$ and all $b_1 \leq \ldots \leq b_D$, one has
\[
\sum_{i=1}^D a_{\sigma(i)} b_i \geq \sum_{i=1}^D a_{D+1-i} b_i.
\]
Since $\psi$ is concave, we observe that the map $z \mapsto \psi(z) - z \psi'(z)$ is monotonically non-decreasing. Without loss of generality, we have $0 < z_1 \leq \ldots \leq z_D$. Furthermore, by the second claim of Lemma \ref{Lj}, the map $i \mapsto j(i)$ is a permutation. Thus, we can apply the rearrangement inequality to obtain
\[
\sum_i  z_{\ii}\left( \psi(z_i) - z_i \psi'(z_i) \right) \geq \sum_i  z_{i'}\left( \psi(z_i) - z_i \psi'(z_i) \right)
\]
with $i':=D+1-i$. Especially, we have $i''=i$. Furthermore, the map $i \mapsto i'$ is a permutation.
Hence,
\begin{eqnarray*}
2 \Gamma_2^\psi (f) (v)  &\geq& \sum_i \psi'(z_i) - \sum_i   z_{\ii} \psi\left( \frac 1 {z_{\ii}} \right) + \sum_i  z_{\ii}\left( \psi(z_i) - z_i \psi'(z_i) \right)\\
&\geq&  \sum_i \psi'(z_i) - \sum_i   \zi \psi\left( \frac 1 {\zi} \right) + \sum_i  \zi\left( \psi(z_i) - z_i \psi'(z_i) \right)\\
&=&    \frac 1 2 \left( \sum_i + \sum_{i'}\right)\left[\psi'(z_i) - \zi \psi\left( \frac 1 {\zi}\right) +  \zi\left( \psi(z_i) - z_i \psi'(z_i) \right)  \right] \\
&=&    \frac 1 2 \sum_i \left[\psi'(z_i) - \zi \psi\left( \frac 1 {\zi}\right) +  \zi\left( \psi(z_i) - z_i \psi'(z_i) \right)  \right]\\
&& +   \frac 1 2 \sum_i \left[\psi'(\zi) - z_i \psi\left( \frac 1 {z_i}\right) +  z_i\left( \psi(\zi) - \zi \psi'(\zi) \right)  \right].\\
\end{eqnarray*}
In the first identity, we used the permutation property of the map $i \mapsto i'$ and its consequence $\sum_i = \sum_{i'}$. 
In the second identity, we used $i''=i$.
Now, we employ the definitions 
$
\widetilde{\psi}(x,y)= \left[\psi'(x) + \psi'(y) \right](1-xy) + x [\psi(y) - \psi(1/x)] + y [\psi(x) - \psi(1/y)]
$
and
$
C_\psi = \inf_{x,y>0} {\widetilde{\psi}(x,y)}/{(\psi(x) + \psi(y) - 2\psi(1))^2}
$
from Definition \ref{DCP} to obtain
\begin{eqnarray*}
\ldots&=& \frac 1 2  \sum_i \widetilde{\psi}(z_i, \zi) \geq \frac 1 2  \sum_i C_\psi \left[ \psi(z_i) + \psi(\zi) \right]^2  \\
&\geq& \frac {C_\psi}{2D} \left[ \sum_i   \psi(z_i) + \psi(\zi) \right]^2  =  \frac {C_\psi}{2D} \left[ 2 \Delta^\psi f (v) \right]^2\\
&=&  \frac { 2 C_\psi}{D} \left[ \Delta^\psi f (v) \right]^2.\\
\end{eqnarray*}

This finishes the proof since $v\in V$ and $f\in C^+(V)$ are arbitrary.  
\end{proof}

  \subsection{Special cases for the function $\psi$}   \label{SSC}
One objective of this subsection is to show that the Harnack inequality on Ricci-flat graphs established in this article is stronger than the one given in \cite[Corollary 5.2]{Bauer2013}.
Moreover, we will discuss that there is a break of analogy in the curvature-dimension condition introduced in \cite{Bauer2013} compared to the $CD$ condition on manifolds. 
From the perspective of our approach, the authors of \cite{Bauer2013} consider the instance $\psi=\sqrt{\cdot}$.
We will give examples for the constants $C_\psi$ and $H_\psi$. Especially, we are interested in the cases $\psi = \log$ and $\psi = \sqrt{\cdot}$.
Furthermore, we will give useful criteria, whether $C_\psi=0$ respectively $H_\psi=\infty$. These cases should be avoided since then, the $CD\psi$ condition for Ricci-flat graphs, respectively the Harnack inequality, degenerates. We will see that $C_\psi = 0$ if $\psi$ does not satisfy a certain symmetry property. Moreover, we will see that $0 < H_\psi < \infty$ if $\psi$ is concave and $\psi''(1)<0$. First, we discuss the constant $H_\psi$ and next, the constant $C_\psi$. Then, we discuss the break of analogy in \cite{Bauer2013} and finally, we give the comparison between the Harnack inequalities.

\begin{lemma}[Degeneration of $H_\psi$] \label{LDH}
Let $\psi \in C^1(\posR)$. 
		If $\psi$ is concave and $\psi''(1)<0$, then $0< H_\psi < \infty$.

\end{lemma}

\begin{proof}
To prove the lemma, we recall the definition of $H_\psi$,
\[
H_\psi = \sup_{x>1} \frac{(\log x)^2}{\overline{\psi}(x)}
\]
with
\[
\overline{\psi}(x)= \psi'(1)\cdot(x-1)  - (\psi(x) - \psi(1)). 
\]
By assumption, $\psi$ is concave and $\psi''(1)<0$. Thus, we also see that $\overline{\psi}$ is concave and $\overline{\psi}''(1)>0$. Additionally, we have $\overline{\psi}'(1)=0$.
Hence, we see $\overline{\psi}(x) > 0$ for $x>1$ and $\overline{\psi}(x) \geq C x$ for some $C>0$ and large $x$. Consequently,
\[
\lim_{x\to \infty} \frac{(\log x)^2}{\overline{\psi}(x)} = 0.
\]
By l'Hopital's rule, we obtain
\[
\lim_{x\to 1} \frac{(\log x)^2}{\overline{\psi}(x)} = \frac 2 {\overline{\psi}''(1)} > 0.
\]
Thus, $H_\psi>0$.

The previous observations guarantee that the function $[1,\infty] \to \IR$, $x \mapsto \frac{(\log x)^2}{\overline{\psi}(x)}$ is continuous. Thus, it attains its maximum and hence, $H_\psi < \infty$.
\end{proof}

\begin{example}[$H_{\log}$ and $H_{\sqrt{\cdot}}$] 
We will prove the identities
\begin{eqnarray*}
H_{\log} &=& 2, \\
H_{\sqrt{\cdot}} &=& 8. \\
\end{eqnarray*}

\begin{proof} 
We start with $H_{\log} = 2$.

The function $(1,\infty) \to \IR$ with $x \mapsto \frac{(\log x)^2}{\overline{\log}(x)}$ is monotonically non-increasing. Thus,
\[
H_{\log} = \sup_{x>1} \frac{(\log x)^2}{\overline{\log}(x)} = \lim_{x\to 1} \frac{(\log x)^2}{\overline{\log}(x)} =  \frac 2 {\overline{\log}''(1)} = 2.
\]

We prove $H_{\sqrt{\cdot}} = 8$.

The function $(1,\infty) \to \IR$ with $x \mapsto \frac{(\log x)^2}{\overline{\sqrt{\cdot}}(x)}$ is monotonically non-increasing. Thus,
\[
H_{\sqrt{\cdot}} = \sup_{x>1} \frac{(\log x)^2}{\overline{\sqrt{\cdot}}(x)} = \lim_{x\to 1} \frac{(\log x)^2}{\overline{\sqrt{\cdot}}(x)} =  \frac 2 {\overline{\sqrt{\cdot}}''(1)} = 8.
\]
This finishes the proof.
\end{proof} 
\end{example}

\begin{lemma}[Degeneration of $C_\psi$] 
Let $\psi \in C^1(\posR)$. 
\eBr
	\item 
	  If $\psi$ is concave, then $C_\psi \geq 0$.
	\item 
		If $\psi(x)+\psi(1/x) \neq 2 \psi(1)$ for some $x>0$, then $C_\psi \leq 0$.
\end{enumerate}
\end{lemma}

\begin{proof}
We recall the definition of $C_\psi$,

\[
C_\psi = \inf_{x,y>0} \frac{\widetilde{\psi}(x,y)}{(\psi(x) + \psi(y) - 2\psi(1))^2}
\]
with
\[
\widetilde{\psi}(x,y)= \left[\psi'(x) + \psi'(y) \right](1-xy) + x [\psi(y) - \psi(1/x)] + y [\psi(x) - \psi(1/y)].
\]
First, we prove \emph{(1)}.
It is sufficient to show $\widetilde{\psi}(x,y) \geq 0$ for all $x,y>0$. For $x,y>0$, we can write 
\[
\widetilde{\psi}(x,y) = y \left[\psi'(x) \left(\frac 1 y - x \right) + \psi(x) - \psi\left( \frac 1 y \right)  \right]
											+	x \left[\psi'(y) \left(\frac 1 x - y \right) + \psi(y) - \psi\left( \frac 1 x \right)  \right].
\]
Since $\psi$ is concave by assumption, we obtain $\psi'(x) \left(\frac 1 y - x \right) + \psi(x) - \psi\left( \frac 1 y \right) \geq 0$ for all $x,y>0$.
Consequently, $\widetilde{\psi}(x,y) \geq 0$ for all $x,y>0$.

Next, we show \emph{(2)}.
We observe $\widetilde{\psi}(x,1/x)=0$ for all $x>0$.
By the assumption, there exists $x>0$, such that $\psi(x)+\psi(1/x) \neq 2 \psi(1)$. 
Hence,
\[
C_\psi \leq  \frac{\widetilde{\psi}(x,1/x)}{(\psi(x) + \psi(1/x) - 2\psi(1))^2} =0.
\]
This finishes the proof.
\end{proof}

\begin{rem}
It would be interesting to know whether the properties concavity of $\psi$ and $\psi(x) + \psi(1/x)= 2\psi(1)$ for all $x>0$, already characterize the case $C_\psi>0$.
\end{rem}

\begin{example}[$C_{\log}$ and $C_{\sqrt{\cdot}}$] 
We will prove 
\begin{eqnarray*}
\frac 1 2 \leq C_{\log} &\leq& 1, \\
C_{\sqrt{\cdot}}&=& 0. \\
\end{eqnarray*}

\begin{rem}
Numerical computations via Mathematica \cite{Wolfram2007} have shown that
$$C_{\log} \approx 0.795.$$
There seems to be no analytic expression for $C_{\log}$.

\end{rem}
\begin{proof} 
First, we show $C_{\sqrt{\cdot}}= 0$.

By the degeneration lemma of $C_\psi$, we obtain $C_{\sqrt{\cdot}}= 0$, since the square root does not satisfy the symmetry condition $\psi(x)+\psi(1/x)=2\psi(1)$ for $x>0$.

Next, we show $\frac 1 2 \leq C_{\log} \leq 1$.

For $x,y>0$ and $\psi = \log$, we observe 
\[
\frac{\widetilde{\psi}(x,y)}{(\psi(x) + \psi(y) - 2\psi(1))^2} = (x+y)\frac{ \frac 1 {xy} - 1  + \log xy}{(\log xy)^2} \geq 2 \sqrt{xy} \frac{ \frac 1 {xy} - 1  + \log xy}{(\log xy)^2}.
\]
The inequality is sharp if $x=y$. Hence, we have
\[
C_{\log} = \inf_{x,y>0} \frac{\widetilde{\psi}(x,y)}{(\psi(x) + \psi(y) - 2\psi(1))^2} = \inf_{z>0} 2 \sqrt{z} \frac{ \frac 1 {z} - 1  + \log z}{(\log z)^2}.
\]

Denote $\varphi(\log x) := \frac {\sqrt{x} (  1/x - 1 + \log x)} {(\log x)^2}$ for $x>0$ and $x\neq 1$.
Then, we obtain
\[
\varphi(x) = \frac {e^{x/2}\left( e^{-x} - 1 + x  \right)}{x^2}
\]
and $C_{\log} = 2 \inf_{x\neq 0} \varphi(x)$.

By l'Hopital's rule, we see the upper bound
\[
\lim_{x\to 0} \varphi(x) = \lim_{x\to 0} \frac {\ddx \left( e^{-x} - 1 + x \right)} {\ddx \left( x^2\right)} = \lim_{x\to 0} \frac {-e^{-x} + 1 } {2x} = \frac 1 2 .
\]
Consequently, $ C_{\log}  \leq 1$.

To prove the lower bound, we write
\begin{eqnarray*}
\varphi(x)  &=& \frac {e^{x/2}\left( e^{-x} - 1 + x  \right)}{x^2}\\
						&=&	\frac {e^{x/2}\left( e^{-x} - 1 + x  \right)}{(e^{x/4} - e^{-x/4})^2} \cdot \frac {(e^{x/4} - e^{-x/4})^2} {x^2} \\
						&=& \frac { e^{-x} - 1 + x  }{(1 - e^{-x/2})^2} \cdot \left( \frac {e^{x/4} - e^{-x/4}} {x} \right)^2.
\end{eqnarray*}
Since $e^{-x/2}\geq 1 - \frac x 2$ and consequently 
\begin{eqnarray*}
 e^{-x} - 1 + x &=&    (e^{-x} + 1)- 2\left(1 - \frac x 2\right)   \\  &\geq& (e^{-x} + 1) - 2e^{-x/2}  \\ &=& (1 - e^{-x/2})^2 
\end{eqnarray*}
we get 
\[
  \frac { e^{-x} - 1 + x  }{(1 - e^{-x/2})^2} \geq 1.
\]

On the other hand, we can compute
\begin{eqnarray*}
\frac {e^{x/4} - e^{-x/4}} {x} &=& \frac 1 x \left[ \sum_{k\geq 0} \frac {\left( x / 4\right)^k}{ k!} - \frac {\left(- x / 4\right)^k} { k!} \right] \\
&=& \frac 2 x \sum_{j \geq 0} \frac {\left( x / 4\right)^{2j+1}}{ (2j+1)!} \\
&=&  \sum_{j \geq 0}  \frac {2x^{2j} }{ (2j+1)! \cdot 4^{2j+1}} \\
&\geq&    \left[ \frac {2x^{2j} }{ (2j+1)! \cdot 4^{2j+1}} \right]_{j=0} \\
&=& \frac 1 2 .
\end{eqnarray*}

Putting these estimates together yields the lower bound of $C_{\log}$
\begin{eqnarray*}
\varphi(x)&=&     \frac { e^{-x} - 1 + x  }{(1 - e^{-x/2})^2} \cdot \left( \frac {e^{x/4} - e^{-x/4}} {x} \right)^2 \\
					&\geq&  1 \cdot \left(\frac 1 2 \right)^2 = \frac 1 4
\end{eqnarray*}
for all $x \neq 0$ and hence, $C_{\log} \geq \frac 1 2$.
\end{proof}

\end{example}

Since $C_{\sqrt{\cdot}}=0$, the $CD\sqrt{\cdot} $ condition degenerates for Ricci-flat graphs. In \cite{Bauer2013}, this problem has been solved by breaking the analogy to the manifolds case.
More specifically, they introduced a weaker form of the $CD\sqrt{\cdot} $ inequality which requires $\Gamma_2^{\sqrt{\cdot}} (f)(x) \geq \frac 1 d \Delta^{\sqrt{\cdot}} f (x)$ only if $\Delta^{\sqrt{\cdot}}f(x) \leq 0$ for all $f \in C^+(V)$ and $x \in V$ with a graph $G=(V,E)$. This additional condition $\Delta^{\sqrt{\cdot}}f(x) \leq 0$ is the break of analogy. There seems to be no possibility to use the semigroup methods from \cite{Bakry2006} to derive a $\psi$-Li Yau inequality from a weak $CD \psi$ condition. But nevertheless, this weak $CD \psi$ condition is sufficient to prove Li-Yau type gradient estimates via the maximum principle.

\subsection{Harnack inequalities on Ricci-flat graphs}

A remarkable result of \cite{Bauer2013} is the Harnack inequality on Ricci-flat graphs.
This states the following. If $G = (V,E)$ is a $D$-Ricci-flat graph, then one has
\[
\log \frac{u(x,T_1)}{u(y,T_2)} \leq D \log \frac{T_2}{T_1} + \frac{4d(x,y)^2}{T_2-T_1} 
\]
for all positive solutions $\sol$ to the heat equation, for all $x,y \in V$ and for all positive $T_1<T_2$.

As claimed in the introduction, we achieve an improvement for Ricci-flat graphs.

\begin{corollary}[Harnack inequality for Ricci-flat graphs] \label{CHR} 
Let $\psi \in C^1(\posR)$ be concave and let $D \in \IN$. If $G = (V,E)$ is a $D$-Ricci-flat graph, then we have
\[
\log \frac{u(x,T_1)}{u(y,T_2)} \leq \frac D {2  C_{\psi}} \log \frac{T_2}{T_1} + \frac{H_\psi d(x,y)^2}{T_2-T_1} 
\]
for all positive solutions $\sol$ to the heat equation, for all $x,y \in V$ and for all positive $T_1<T_2$.
\end{corollary}
\begin{proof}
This is an easy consequence of the Harnack inequality (Theorem \ref{THarnack}) and the $CD\psi$ condition for Ricci-flat graphs (Theorem \ref{TRicci}). 
\end{proof}

If we choose $\psi = \log$, then by using 
$C_{\log} \geq \frac 1 2$ and $H_{\log}=2$ (cf. Subsection \ref{SSC}), we obtain
\[
\log \frac{u(x,T_1)}{u(y,T_2)} \leq  D \log \frac{T_2}{T_1} + \frac{2 d(x,y)^2}{T_2-T_1} .
\]
Using the numerical result $C_{\log} \approx 0.795$, the previous corollary improves to
\[
\log \frac{u(x,T_1)}{u(y,T_2)} \leq  0.629 D \log \frac{T_2}{T_1} + \frac{2 d(x,y)^2}{T_2-T_1} .
\]

This means, our upper bound is by a factor of 1.59 smaller than the one obtained in \cite{Bauer2013}.


Florentin Münch,  \\
Mathematisches Institut, Friedrich Schiller Universität Jena, \\
D-07745 Jena, Germany\\
\texttt{florentin.muench@uni-jena.de}

\end{document}